\documentclass[12pt]{article}

\usepackage[utf8]{inputenc}
\usepackage[a4paper,nomarginpar]{geometry}
\usepackage[english]{babel}
\usepackage{enumitem}
\usepackage{amsmath,amsfonts,amssymb,amsthm,amscd}
\usepackage{tikz,xcolor}

\allowdisplaybreaks
\theoremstyle{plain}
\newtheorem{theorem}{Theorem}%[section]
\newtheorem{lemma}[theorem]{Lemma}%[section]
\newtheorem{corollary}[theorem]{Corollary}%[section]

\theoremstyle{definition}
\newtheorem{remark}[theorem]{Remark}%[section]
\newtheorem{example}[theorem]{Example}%[section]
\newcommand{\N}{\mathbb{N}} %% Conjunto naturales:     \N
\newcommand{\Z}{\mathbb{Z}} %% Conjunto enteros:       \Z
\newcommand{\R}{\mathbb{R}} %% Conjunto reales:        \R
\newcommand{\C}{\mathbb{C}} %% Conjunto complejos:     \C
\newcommand{\K}{\mathbb{K}} %% Cuerpo		           \K

\newcommand{\cC}{\mathcal{C}}
\newcommand{\cD}{\mathcal{D}}

\newcommand{\cP}{\mathcal{P}}
\newcommand{\cR}{\mathcal{R}}

\newcommand{\fM}{\mathfrak{M}}
\newcommand{\fX}{\mathfrak{X}}

\newcommand{\card}{\operatorname{card}}
\newcommand{\supp}{\operatorname{supp}}
\newcommand{\ldens}{\operatorname{\underline{dens}}}
\newcommand{\udens}{\operatorname{\overline{dens}}}

\newcommand{\spa}{\operatorname{span}}      %% Func. env. lineal:        \spa

\newcommand{\esssup}{\operatorname{ess\, sup}}
\newcommand{\eps}{\varepsilon}
\newcommand{\ov}{\overline}

\DeclareRobustCommand{\rchi}{{\mathpalette\irchi\relax}}
\newcommand{\irchi}[2]{\raisebox{\depth}{$#1\chi$}}

\title{On the dynamics of weighted composition operators}

\author{Nilson C. Bernardes Jr., Antonio Bonilla, Jo\~ao V. A. Pinto}

\date{}

\begin{document}

\maketitle

\begin{abstract}
We study the properties of power-boundedness, Li-Yorke chaos, distributional chaos, absolutely Ces\`aro boundedness
and mean Li-Yorke chaos for weighted composition operators on $L^p(\mu)$ spaces and on $C_0(\Omega)$ spaces.
We illustrate the general results by presenting several applications to weighted shifts on the classical sequence spaces
$c_0(\N)$, $c_0(\Z)$, $\ell^p(\N)$ and $\ell^p(\Z)$ ($1 \leq p < \infty$) and to weighted translation operators on the
classical function spaces $C_0[1,\infty)$, $C_0(\R)$, $L^p[1,\infty)$ and $L^p(\R)$ ($1 \leq p < \infty$).
\end{abstract}



%%%%%%%%%%%%%%%%%%%%%%%%%%%%%%%%%%%%%%%%%%%%%%%%%%%%%%%%%%%%%%%%%%%%%%
%%%%%%%%%%%%%%%%%%%%%%%%%%%%%%%%%%%%%%%%%%%%%%%%%%%%%%%%%%%%%%%%%%%%%%

\section{Introduction}

Linear dynamics is a branch of mathematics that lies at the interface between the big areas 
of dynamical systems and operator theory.
Its main focus is to study the dynamics of continuous linear operators on topological vector spaces (often Banach or Fr\'echet spaces).
The area has undergone significant development during the last four decades.
We refer the reader to the books \cite{BaMa,GP} for an overview of the area up to around 2010.
The main objective of these books is the study of chaotic behaviors related to the concept of hypercyclicity (existence of a dense orbit), 
such as hypercyclicity itself, Devaney chaos, mixing and frequent hypercyclicity.
During the 2010s, the series of papers \cite{BBMP11,BBMP,BBMP2,BBPW,BBP} laid the foundation for a theory of chaotic behaviors 
related to the dynamics of pairs of points (including Li-Yorke chaos, distributional chaos and mean Li-Yorke chaos) 
in the context of linear dynamics.
More recently, some fundamental notions in dynamical systems that are not notions of chaos, such as hyperbolicity, expansivity, 
shadowing and stability, have started to be systematically investigated in the setting of linear dynamics.
We refer the reader to the recent articles \cite{BCDFP,BerPer}, where many additional references can be found.

On the other hand, composition operators constitute a very important class of operators in operator theory and its applications.
There is a vast literature on the dynamics of these operators in different contexts (analytic, measure-theoretic, topological).
We refer the reader to \cite{AJK,BonKalPer21,BonDAnDarPia22,PBouJSha97,EGalAMon04,TKal19,Prz,JSha93}, for instance.
In the particular case of composition operators on $L^p ( \mu)$ spaces, Li-Yorke chaos was studied in \cite{BerDarPi,BV}, 
topological transitivity and mixing in \cite{BaDarPi}, Devaney chaos and frequent hypercyclicity in \cite{DarPi}, 
generalized hyperbolicity and shadowing in \cite{DAnDarMai21}, expansivity and strong structural stability in \cite{M}, 
distributional chaos in \cite{HeYin}, and Kitai's Criterion in \cite{GG}.

Our goal in the present work is to study the concepts of power-boundedness, Li-Yorke chaos, distributional chaos,
$p$-absolutely Ces\`aro boundedness and mean Li-Yorke chaos for weighted composition operators 
\[
C_{w,f}(\varphi) = (\varphi \circ f) \cdot w
\]
on the classical Banach spaces $L^p(\mu)$ and $C_0(\Omega)$. 
Special emphasis will be given to obtaining characterizations of these concepts.
In order to illustrate the general results, we will present several applications to two special classes of weighted composition operators,
namely, weighted shifts on the classical sequence spaces $c_0(\N)$, $c_0(\Z)$, $\ell^p(\N)$ and $\ell^p(\Z)$ ($1 \leq p < \infty$)
and weighted translation operators on the classical function spaces $C_0[1,\infty)$, $C_0(\R)$, $L^p[1,\infty)$ and $L^p(\R)$
($1 \leq p < \infty$).

Our results on Li-Yorke chaos will emphasize the close relationship between this concept and the notion of power-boundedness
and will complement previous studies developed in \cite{BerDarPi,BV}.
In the case of $L^p(\mu)$ spaces, our results on distributional chaos for weighted composition operators will extend and complement 
recent results obtained in the preprint \cite{HeYin} in the unweighted case, but we will also study here the case of $C_0(\Omega)$ spaces.
To the best of the authors' knowledge, the concepts of absolute Ces\`aro boundedness and mean Li-Yorke chaos have not been studied
before in the context of the present article.

Throughout $\K$ denotes either the field $\R$ of real numbers or the field $\C$ of complex numbers,
$\Z$ denotes the ring of integers, $\N$ denotes the set of all positive integers, and $\N_0 = \N \cup \{0\}$.
By an {\em operator} on a Banach space $Y$, we mean a bounded linear map $T : Y \to Y$.
Recall that the {\em operator norm} of such a map is the number $\|T\| = \sup\{\|Ty\| : \|y\| \leq 1\}$.

%%%%%%%%%%%%%%%%%%%%%%%%%%%%%%%%%%%%%%%%%%%%%%%%%%%%%%%%%%%%%%%%%%%%%%
%%%%%%%%%%%%%%%%%%%%%%%%%%%%%%%%%%%%%%%%%%%%%%%%%%%%%%%%%%%%%%%%%%%%%%

\section{Preliminaries}

\subsection{Weighted composition operators on $L^p(\mu)$}

Throughout this article we fix a real number $p \in [1,\infty)$ and an arbitrary positive measure space $(X,\fM,\mu)$,
unless otherwise specified. 
$L^p(\mu)$ denotes the Banach space over $\K$ of all $\K$-valued $p$-integrable functions on $(X,\fM,\mu)$ 
endowed with the $p$-norm
\[
\|\varphi\|_p = \Big(\int_X |\varphi|^p d\mu\Big)^{\frac{1}{p}}.
\]
$L^\infty(\mu)$ denotes the Banach space over $\K$ of all $\K$-valued essentially bounded measurable functions on $(X,\fM,\mu)$
endowed with the essential supremum norm
\[
\|\varphi\|_\infty = \esssup |f|.
\]
We also consider a measurable map $w : X \to \K$ such that
\begin{equation}\label{weight}
\varphi \cdot w \in L^p(\mu) \ \text{ for all } \varphi \in L^p(\mu).
\end{equation}
If the measure $\mu$ is semifinite, that is, every set with infinite measure contains a set with positive finite measure 
(in particular, if $\mu$ is $\sigma$-finite), then (\ref{weight}) is equivalent to $w \in L^{\infty}(\mu)$ \cite [Proposition~7]{BV},
but this equivalence is not true in general \cite[Remark~8]{BV}.
Given a bimeasurable map $f : X \to X$ (i.e., $f(B) \in \fM$ and $f^{-1}(B) \in \fM$ whenever $B \in \fM$), 
it is not difficult to show that the {\em weighted composition operator}
\[
C_{w,f}(\varphi) = (\varphi \circ f) \cdot w
\]
is a well-defined bounded linear operator on $L^p(\mu)$ if and only if there is a constant $c \in (0,\infty)$ such that
\begin{equation}\label{condition}
\int_B |w|^p d\mu \leq c\, \mu(f(B)) \ \text{ for every } B \in \fM.
\end{equation}
Moreover, in this case, $\|C_{w,f}\|^p \leq c$. Whenever we consider a weighted composition operator $C_{w,f}$ on $L^p(\mu)$, 
we will implicitly assume that $w$ and $f$ satisfy conditions (\ref{weight}) and (\ref{condition}).
As in \cite[Section~3]{BV}, we associate to $w$ and $f$ the following sequence of positive measures on $(X,\fM)$:
\[
\mu_n(B) = \int_B |w^{(n)}|^p d\mu \ \ \ \ (B \in \fM, n \in \N),
\]
where
\[
w^{(1)} = w \ \ \ \text{ and } \ \ \ w^{(n)} = (w \circ f^{n-1}) \cdot\ldots\cdot (w \circ f) \cdot w \ \text{ for } n \geq 2.
\]

\begin{lemma}[\cite{BV}]\label{null} 
If (\ref{condition}) holds, then
\[
\mu_n(B) \leq c^n \mu(f^n(B)) \ \text{ for all } B \in \fM \text{ and } n \in \N.
\]
\end{lemma}

\begin{example}
If $X = \N$ (resp.\ $X = \Z$), $\fM = \cP(X)$ (the power set of $X$), $\mu$ is the counting measure on $\fM$ 
and $f : n \in X \mapsto n+1 \in X$, then $L^p(\mu) = \ell^p(\N)$ (resp.\ $L^p(\mu) = \ell^p(\Z)$)
and $C_{w,f}$ coincides with the {\em unilateral} (resp.\ {\em bilateral}) {\em weighted backward shift}
\[
B_w : (x_n)_{n \in X} \in \ell^p(X) \mapsto (w_nx_{n+1})_{n \in X} \in \ell^p(X).
\]
\end{example}

\begin{example}
If $X = [1,\infty)$ (resp.\ $X = \R$), $\fM$ is the $\sigma$-algebra of all Lebesgue measurable sets in $X$, 
$\mu$ is the Lebesgue measure on $\fM$ and $f : x \in X \mapsto x+1 \in X$, then 
$L^p(\mu) = L^p[1,\infty)$ (resp.\ $L^p(\mu) = L^p(\R)$)
and $C_{w,f}$ coincides with the {\em unilateral} (resp.\ {\em bilateral}) {\em weighted translation operator}
\[
T_w : \varphi \in L^p(X) \mapsto \varphi(\cdot + 1) w(\cdot) \in L^p(X).
\]
\end{example}

Recall that $B_w$ (resp.\ $T_w$) is a well-defined bounded linear operator exactly when $w$ is a bounded sequence 
(resp.\ an essentially bounded measurable function).
According to our convention, whenever we consider such an operator $B_w$ (resp.\ $T_w$), we will implicitly assume that this is the case.

%%%%%%%%%%%%%%%%%%%%%%%%%%%%%%
%%%%%%%%%%%%%%%%%%%%%%%%%%%%%%

\subsection{Weighted composition operators on $C_0(\Omega)$}

Throughout this article we fix an arbitrary locally compact Hausdorff space $\Omega$, unless otherwise specified.
$C_0(\Omega)$ denotes the Banach space over $\K$ of all continuous maps $\varphi : \Omega \to \K$ that vanish at infinity 
endowed with the supremum norm 
\[
\|\varphi\| = \sup_{x \in \Omega} |\varphi(x)|.
\]
More generally, given any bounded map $\varphi : \Omega \to \K$ and any $B \subset \Omega$, we define
\[
\|\varphi\|_B = \sup_{x \in B} |\varphi(x)|,
\]
where we consider this supremum to be $0$ if $B = \varnothing$.
In the case $B = \Omega$, we usually write $\|\varphi\|$ instead of $\|\varphi\|_\Omega$ even if $\varphi \not\in C_0(\Omega)$.
Recall that the {\em support} of $\varphi : \Omega \to \K$ is defined by 
\[
\supp \varphi = \ov{\{x \in \Omega : \varphi(x) \neq 0\}}.
\]
$C_c(\Omega)$ denotes the subspace of $C_0(\Omega)$ consisting of those $\varphi \in C_0(\Omega)$
that have compact support. We also consider a continuous map $w : \Omega \to \K$ such that
\begin{equation}\label{weightC0X}
\varphi \cdot w \in C_0(\Omega) \ \text{ for all } \varphi \in C_0(\Omega).
\end{equation}
It is easy to show that (\ref{weightC0X}) holds if and only if $w$ is bounded in $\Omega$.
Given a continuous map $f : \Omega \to \Omega$, it is not difficult to show that the {\em weighted composition operator}
\[
C_{w,f}(\varphi) = (\varphi \circ f) \cdot w
\]
is a well-defined bounded linear operator on $C_0(\Omega)$ if and only if
\begin{equation}\label{conditionC0X}
\{x \in \Omega : f(x) \in K \text{ and } |w(x)| \geq \eps\} \text{ is compact, for all } \eps > 0 \text{ and } K \subset \Omega \text{ compact}.
\end{equation}
Whenever we consider a weighted composition operator $C_{w,f}$ on $C_0(\Omega)$, 
we will implicitly assume that $w$ and $f$ satisfy conditions (\ref{weightC0X}) and (\ref{conditionC0X}).
As in \cite[Section~2]{BV}, we associate to $w$ and $f$ the following sequence of continuous maps from $\Omega$ into $\K$:
\[
w^{(1)} = w \ \ \ \text{ and } \ \ \ w^{(n)} = (w \circ f^{n-1}) \cdot\ldots\cdot (w \circ f) \cdot w \ \text{ for } n \geq 2.
\]

\begin{example}
If $\Omega = \N$ (resp.\ $\Omega = \Z$) endowed with the discrete topology and $f : n \in \Omega \mapsto n+1 \in \Omega$, 
then $C_0(\Omega) = c_0(\N)$ (resp.\ $C_0(\Omega) = c_0(\Z)$) 
and $C_{w,f}$ coincides with the {\em unilateral} (resp.\ {\em bilateral}) {\em weighted backward shift}
\[
B_w : (x_n)_{n \in \Omega} \in c_0(\Omega) \mapsto (w_nx_{n+1})_{n \in \Omega} \in c_0(\Omega).
\]
\end{example}

\begin{example}
If $\Omega = [1,\infty)$ (resp.\ $\Omega = \R$) endowed with its usual topology and $f : x \in \Omega \mapsto x+1 \in \Omega$, 
then $C_0(\Omega) = C_0[1,\infty)$ (resp.\ $C_0(\Omega) = C_0(\R)$)
and $C_{w,f}$ coincides with the {\em unilateral} (resp.\ {\em bilateral}) {\em weighted translation operator}
\[
T_w : \varphi \in C_0(\Omega) \mapsto \varphi(\cdot + 1) w(\cdot) \in C_0(\Omega).
\]
\end{example}

Recall that $B_w$ (resp.\ $T_w$) is a well-defined bounded linear operator exactly when $w$ is a bounded sequence 
(resp.\ a bounded continuous function).
According to our convention, whenever we consider such an operator $B_w$ (resp.\ $T_w$), we will implicitly assume that this is the case.

%%%%%%%%%%%%%%%%%%%%%%%%%%%%%%%%%%%%%%%%%%%%%%%%%%%%%%%%%%%%%%%%%%%%%%
%%%%%%%%%%%%%%%%%%%%%%%%%%%%%%%%%%%%%%%%%%%%%%%%%%%%%%%%%%%%%%%%%%%%%%

\section{Power-bounded versus Li-Yorke chaotic weighted composition operators}

Recall that an operator $T$ on a Banach space $Y$ is said to be {\em power-bounded} if 
there exists $C \in (0,\infty)$ such that $\|T^n\| \leq C$ for all $n \in \N$.
More generally, given a subspace $Z$ of $Y$, we say that $T$ is {\em power-bounded in $Z$} if 
there exists $C \in (0,\infty)$ such that
\[
\|T^n|_Z\| \leq C \ \text{ for all } n \in \N,
\]
where $T^n|_Z : Z \to Y$ is the bounded linear map obtained by restricting the domain of $T^n$ to $Z$.

Given a metric space $M$, recall that a map $f : M \to M$ is said to be {\em Li-Yorke chaotic} if there exists an uncountable set 
$S \subset M$ such that each pair $(x,y)$ of distinct points in $S$ is a {\em Li-Yorke pair for $f$}, in the sense that
\[
\liminf_{n \to \infty} d(f^n(x),f^n(y)) = 0 \ \ \text{ and } \ \ \limsup_{n \to \infty} d(f^n(x),f^n(y)) > 0.
\]
If the set $S$ can be chosen to be dense in $M$, then $f$ is {\em densely Li-Yorke chaotic}.

An extensive study of the concept of Li-Yorke chaos in the setting of linear dynamics was developed in \cite{BBMP11,BBMP2}.
In particular, the following useful characterizations were obtained:
For any operator $T$ on any Banach space $Y$, the following assertions are equivalent:
\begin{itemize}
\item [(i)] $T$ is Li-Yorke chaotic;
\item [(ii)] $T$ admits a {\em semi-irregular vector}, that is, a vector $y \in Y$ such that
  \[
  \liminf_{n \to \infty} \|T^ny\| = 0 \ \ \text{ and } \ \ \limsup_{n \to \infty} \|T^ny\| > 0.
  \]
\item [(iii)] $T$ admits an {\em irregular vector}, that is, a vector $y \in Y$ such that
  \[
  \liminf_{n \to \infty} \|T^ny\| = 0 \ \ \text{ and } \ \ \limsup_{n \to \infty} \|T^ny\| = \infty.
  \]
\end{itemize}

Let us also recall that by an {\em irregular manifold for $T$} we mean a vector subspace of $Y$ consisting,
except for the zero vector, of irregular vectors for $T$.

%%%%%%%%%%%%%%%%%%%%%%%%%%%%%%
%%%%%%%%%%%%%%%%%%%%%%%%%%%%%%

\subsection{The case of the space $L^p(\mu)$}

\begin{theorem}\label{powerbounded}
Consider a weighted composition operator $C_{w,f}$ on $L^p(\mu)$. For each $n \in \N$,
\begin{equation}\label{power0}
\|(C_{w,f})^n\| = \sup_{0 < \mu(B) < \infty} \Big(\frac{\mu_n(f^{-n}(B))}{\mu(B)}\Big)^\frac{1}{p},
\end{equation}
where the supremum is taken over all measurable sets $B$ satisfying $0 < \mu(B) < \infty$.
In particular, $C_{w,f}$ is power-bounded if and only if there exists a constant $C \in (0,\infty)$ such that, 
for each measurable set $B$ of finite positive $\mu$-measure,
\begin{equation}\label{power}
\mu_n(f^{-n}(B)) \le C\, \mu (B) \ \text{ for all } n \in \N.
\end{equation}
\end{theorem}

\begin{proof}
Fix $n \in \N$ and denote the right-hand side of (\ref{power0}) by $r$.

Given a measurable set $B$ of finite positive $\mu$-measure, define $\phi = \frac{1}{\mu(B)^{1/p}}\,\rchi_B$. 
Then $\|\phi\|_p =1$ and
\begin{eqnarray*}
\|(C_{w,f})^n(\phi)\|_p^p & = &\frac{1}{\mu(B)} \int_{X}|\rchi_{B}\circ f^n|^p|w\circ f^{n-1}|^p\cdots|w\circ f|^p|w|^p d\mu\\
  & = &\frac{1}{\mu(B)} \int_{f^{-n}(B)}|w\circ f^{n-1}|^p\cdots|w\circ f|^p|w|^p d\mu\\
  & = &\frac{\mu_n( f^{-n}(B))}{\mu (B)}\cdot
\end{eqnarray*}
This implies that $r \leq \|(C_{w,f})^n\|$.

On the other hand, take any real number $t > 1$. Given $\varphi \in L^p(\mu)$, $ \varphi \ne 0$, consider the measurable sets
\[
B_i = \{x \in X : t^{i-1} \le |\varphi(x)| < t^i\} \ \ \ (i \in \Z).
\]
Since
\[
\sum_{i \in \Z} t^{(i-1)p} \mu(B_i) \le \int_X |\varphi|^p d\mu < \infty,
\]
we have that $\mu(B_i)< \infty$ for all $i \in \Z$. Since $\mu_n(f^{-n}(B)) = 0$ if $\mu(B) = 0$ (by Lemma~\ref{null}),
it follows from the definition of $r$ that $\mu_n( f^{-n}(B_i)) \le r^p \mu(B_i)$ for all $i \in \Z$. Therefore,
\begin{eqnarray*}
\|(C_{w,f})^n(\varphi)\|_p^p & = & \sum_{i \in \Z} \int_{f^{-n}(B_i)} |\varphi\circ f^n|^p|w\circ f^{n-1}|^p\cdots|w\circ f|^p|w|^p d\mu\\
  & \le & \sum_{i \in \Z} t^{ip} \int_{f^{-n}(B_i)} |w\circ f^{n-1}|^p\cdots|w\circ f|^p|w|^p d\mu\\
  & =   & \sum_{i \in \Z} t^{ip} \mu_n( f^{-n}(B_i)) \le t^p r^p \sum_{i \in \Z} t^{(i-1)p} \mu(B_i) \le t^p r^p \|\varphi\|_p^p.
\end{eqnarray*}
This implies that $\|(C_{w,f})^n\| \leq t\, r$. Since $t > 1$ is arbitrary, we obtain $\|(C_{w,f})^n\| \leq r$,
which completes the proof of (\ref{power0}).

The last assertion follows from (\ref{power0}).
\end{proof}

In the particular case of a weighted shift $B_w$ on $\ell^p(\N)$, it is well known that
\begin{equation}\label{WS-1}
\|(B_w)^n\| = \sup_{i \in \N} |w_i \cdots w_{i+n-1}| \ \text{ for all } n \in \N,
\end{equation}
which is a special case of formula (\ref{power0}). Thus, $B_w$ is power-bounded if and only if
\begin{equation}\label{WS-2}
\sup\{|w_i \cdots w_j| : i,j \in \N, i \le j\} < \infty.
\end{equation}
Therefore, \cite[Corollary 14]{BV} (see also \cite[Proposition~27]{BBMP11}) can be rewritten as follows.

\begin{corollary}
A weighted shift $B_w$ on $\ell^p(\N)$ is Li-Yorke chaotic if and only if it is not power-bounded.
\end{corollary}

Similarly, in the particular case of a weighted shift $B_w$ on $\ell^p(\Z)$, it is well known that
\begin{equation}\label{WS-3}
\|(B_w)^n\| = \sup_{i \in \Z} |w_i \cdots w_{i+n-1}| \ \text{ for all } n \in \N,
\end{equation}
which is also a special case of formula (\ref{power0}). Thus, $B_w$ is power-bounded if and only if
\begin{equation}\label{WS-4}
\sup\{|w_i \cdots w_j| : i,j \in \Z, i \le j\} < \infty.
\end{equation}
Hence, \cite[Corollary 15]{BV} (see also \cite[Corollary~1.6]{BerDarPi}) can be rewritten as follows.

\begin{corollary}
A weighted shift $B_w$ on $\ell^p(\Z)$ with nonzero weights is Li-Yorke chaotic if and only if 
it is not power-bounded and $\displaystyle \liminf_{n \to \infty} |w_{-n} \cdots w_{-1}| = 0$.
\end{corollary}

We now return to general weighted composition operators on $L^p(\mu)$.

\begin{theorem}\label{LY-PB}
Given a weighted composition operator $C_{w,f}$ on $L^p(\mu)$, let $\fX$ be the closed subspace of $L^p(\mu)$ 
generated by the characteristic functions $\rchi_B$ such that $B$ is a measurable set of finite $\mu$-measure satisfying
\begin{equation}\label{LY-PB-1}
\liminf_{n \to \infty} \mu_n(f^{-n}(B)) =0.
\end{equation}
Then, $C_{w,f}$ is Li-Yorke chaotic if and only if it is not power-bounded in $\fX$.
\end{theorem}

\begin{proof}
Suppose that $C_{w,f}$ is Li-Yorke chaotic. By \cite[Theorem~10]{BV}, there exist a nonempty countable family $(B_i)_{i \in I}$
of measurable sets of finite positive $\mu$-measure and an increasing sequence $(n_j)_{j \in \N}$ of positive integers such that
\begin{equation}\label{LY-PB-2}
\sup\Big\{\frac{\mu_n(f^{-n}(B_i))}{\mu(B_i)} : i \in I, n \in \N\Big\} = \infty
\end{equation}
and
\begin{equation}\label{LY-PB-3}
\lim_{j \to \infty} \mu_{n_j}(f^{-n_j}(B_i)) = 0 \ \text{ for all } i \in I.
\end{equation}
For each $i \in I$, let $\phi_i = \frac{1}{\mu(B_i)^{1/p}}\, \rchi_{B_i}$. 
By (\ref{LY-PB-3}), $\phi_i \in \fX$ for all $i \in I$. Since $\|\phi_i\|_p = 1$ and 
\[
\|(C_{w,f})^n(\phi_i)\|_p^p = \frac{\mu_n(f^{-n}(B_i))}{\mu(B_i)} \ \ \ (i \in I, n \in \N),
\]
it follows from (\ref{LY-PB-2}) that $C_{w,f}$ is not power-bounded in $\fX$.

Conversely, suppose that $C_{w,f}$ is not Li-Yorke chaotic. Then, $C_{w,f}$ does not admit a semi-irregular vector.
Thus, (\ref{LY-PB-1}) is equivalent to
\[
\lim_{n \to \infty} \mu_n(f^{-n}(B)) =0.
\]
This implies that the set $\cR_1$ of all $\varphi \in \fX$ whose orbit under $C_{w,f}$ 
has a subsequence converging to zero is dense in $\fX$.
Hence, by \cite[Corollary~4]{BBMP2}, $\cR_1$ is residual in $\fX$.
If $C_{w,f}$ was not power-bounded in $\fX$, the Banach-Steinhaus theorem \cite[Theorem~2.5]{WRud91} would imply that 
the set $\cR_2$ of all $\varphi \in \fX$ whose orbit under $C_{w,f}$ is unbounded is also residual in $\fX$.
This would imply the existence of an irregular vector for $C_{w,f}$ (take any vector in $\cR_1 \cap \cR_2$).
Since we are assuming that $C_{w,f}$ is not Li-Yorke chaotic, we conclude that $C_{w,f}$ must be power-bounded in~$\fX$.
\end{proof}

\begin{theorem}\label{DLY-PB}
Consider a weighted composition operator $C_{w,f}$ on $L^p(\mu)$. 
If for every measurable set $A$ of finite $\mu$-measure and for every $\eps > 0$, there is a measurable set $B \subset A$ with 
\[
\mu(A \backslash B) < \eps \ \ \text{ and } \ \ \ \liminf_{n \to \infty} \mu_n(f^{-n}(B)) = 0,
\]
then the following assertions are equivalent:
\begin{itemize}
\item [\rm (i)] $C_{w,f}$ is Li-Yorke chaotic;
\item [\rm (ii)] $C_{w,f}$ has a residual set of irregular vectors;
\item [\rm (iii)] $C_{w,f}$ is not power-bounded.
\end{itemize}
\end{theorem}

\begin {proof} 
Let $\fX_0$ be the set of all simple functions of the form $\sum _{k=1}^m b_k \rchi_{B_k}$, 
where $b_1,\ldots,b_m$ are scalars and $B_1,\ldots,B_m$ are measurable sets of finite $\mu$-measure satisfying
\begin{equation}\label{DLY-PB-1}
\liminf_{n \to \infty} \mu_n(f^{-n}(B_1 \cup \ldots \cup B_m)) = 0.
\end{equation}
The assumption implies that $\fX_0$ is dense in $L^p(\mu)$. Moreover, by (\ref{DLY-PB-1}),
\[
\liminf_{n \to \infty} (C_{w,f})^n(\varphi) = 0 \ \text{ for all } \varphi \in \fX_0.
\]
Hence, the set $\cR_1$ of all $\varphi \in L^p(\mu)$ whose orbit under $C_{w,f}$ 
has a subsequence converging to zero is residual in $L^p(\mu)$.
If (iii) holds, then the set $\cR_2$ of all $\varphi \in L^p(\mu)$ whose orbit under $C_{w,f}$ 
is unbounded is also residual in $L^p(\mu)$, which implies (ii).
The other implications are clear.
\end{proof}

\begin{remark}\label{R-DLY-PB}
{\bf (a)} In the case $L^p(\mu)$ is separable, \cite[Theorem~10]{BBMP2} says that (ii) is equivalent to
\begin{itemize}
\item [(ii')] $C_{w,f}$ is densely Li-Yorke chaotic.
\end{itemize}

\noindent
{\bf (b)} If the space $L^p(\mu)$ is separable and there is an increasing sequence $(n_j)_{j \in \N}$ of positive integers such that,
for every measurable set $A$ of finite $\mu$-measure and for every $\eps > 0$, there is a measurable set $B \subset A$ with 
\[
\mu(A \backslash B) < \eps \ \ \text{ and } \ \ \ \lim _{j \to \infty} \mu_{n_j}(f^{-n_j}(B)) = 0,
\]
then \cite[Theorem~31]{BBMP2} implies that (i)--(iii) are also equivalent to
\begin{itemize}
\item [(iv)] $C_{w,f}$ admits a dense irregular manifold.
\end{itemize}
\end{remark}

Let us now present some applications to weighted translation operators.

\begin{corollary}
Let $X = [1,\infty)$ or $\R$. For any weighted translation operator $T_w$ on $L^p(X)$,
\[
\|(T_w)^n\| = \|w^{(n)}\|_\infty \ \text{ for all } n \in \N.
\]
In particular, $T_w$ is power-bounded if and only if $\sup_{n \in \N} \|w^{(n)}\|_\infty < \infty$.
\end{corollary}

\begin{proof}
Fix $n \in \N$. In the present case, $\mu$ is the Lebesgue measure (which is translation invariant) and $f : x \in X \mapsto x+1 \in X$. 
By Theorem~\ref{powerbounded},
\[
\|(T_w)^n\| = \sup_{0 < \mu(B) < \infty} \Big(\frac{\mu_n(f^{-n}(B))}{\mu(B)}\Big)^\frac{1}{p}
\leq \sup_{0 < \mu(B) < \infty} \Big(\frac{\|w^{(n)}\|_\infty^p\, \mu(f^{-n}(B))}{\mu(B)}\Big)^\frac{1}{p} = \|w^{(n)}\|_\infty.
\]
On the other hand, take $0 < \delta < \|w^{(n)}\|_\infty$ (the case $\|w^{(n)}\|_\infty = 0$ is trivial) and define
\[
B = \{x \in X : |w^{(n)}(x)| \geq \|w^{(n)}\|_\infty - \delta\}.
\]
Since $\mu(B) > 0$, we can choose $k \in \N$ such that the set $B' = B \cap [-k,k]$ has positive $\mu$-measure.
Hence, by Theorem~\ref{powerbounded},
\[
\|(T_w)^n\| \geq \Big(\frac{\mu_n(f^{-n}(B' + n))}{\mu(B' + n)}\Big)^\frac{1}{p}
= \Big(\frac{\mu_n(B')}{\mu(B')}\Big)^\frac{1}{p} \geq \|w^{(n)}\|_\infty - \delta.
\]
By letting $\delta \to 0^+$, we obtain $\|(T_w)^n\| \geq \|w^{(n)}\|_\infty$.
\end{proof}

\begin{corollary}
For weighted shifts $B_w$ on $\ell^p(\N)$ and weighted translation operators $T_w$ on $L^p[1,\infty)$,
properties (i)--(iv) above are always equivalent to each other.
\end{corollary}

\begin{proof}
It is enough to observe that the conditions in Remark~\ref{R-DLY-PB}(b) are satisfied 
with $(n_j)_{j \in \N}$ being the full sequence $(n)_{n \in \N}$.
\end{proof}

\begin{corollary}
Given a weighted translation operator $T_w$ on $L^p(\R)$, let $\fX$ be the closed subspace of $L^p(\R)$ generated by the 
characteristic functions $\rchi_B$ such that $B$ is a Lebesgue measurable set in $\R$ of finite Lebesgue measure satisfying
\[
\liminf_{n \to \infty} \int_B |w(x-n) \cdots w(x-1)|^p dx = 0.
\]
Then, the following assertions are equivalent:
\begin{itemize}
\item [\rm (i)] $T_w$ is Li-Yorke chaotic;
\item [\rm (ii)] There exists a nonempty countable family $(B_i)_{i \in I}$ of Lebesgue measurable sets in $\R$ of finite positive 
  Lebesgue measure such that
  \[
  \liminf_{n \to \infty} \int_{B_i} |w(x-n) \cdots w(x-1)|^p dx = 0 \ \text{ for all } i \in I
  \]
  and
  \[
  \sup\Big\{\frac{1}{\mu(B_i)} \int_{B_i} |w(x-n) \cdots w(x-1)|^p dx : i \in I, n \in \N\Big\} = \infty.
  \]
\item [\rm (iii)] $T_w$ is not power-bounded in $\fX$.
\end{itemize}
\end{corollary}

\begin{proof}
The equivalence (i) $\Leftrightarrow$ (ii) follows from \cite[Theorem~10]{BV} and \cite[Remark~11]{BV}.
The equivalence (i) $\Leftrightarrow$ (iii) is a restatement of Theorem~\ref{LY-PB} in the present particular situation.
\end{proof}

%%%%%%%%%%%%%%%%%%%%%%%%%%%%%%
%%%%%%%%%%%%%%%%%%%%%%%%%%%%%%

\subsection{The case of the space $C_0(\Omega)$}

\begin{theorem}\label{PB-C0X}
For any weighted composition operator $C_{w,f}$ on $C_0(\Omega)$,
\begin{equation}\label{PB-C0X-1}
\|(C_{w,f})^n\| = \|w^{(n)}\| \ \text{ for all } n \in \N.
\end{equation}
In particular, $C_{w,f}$ is power-bounded if and only if there is a constant $C \in (0,\infty)$ such that
\begin{equation}\label{PB-C0X-2}
\|w^{(n)}\| \leq C \ \text{ for all } n \in \N.
\end{equation}
\end{theorem}

\begin{proof}
Since
\[
\|(C_{w,f})^n(\varphi)\| = \|(\varphi \circ f^n) \cdot w^{(n)}\| \leq \|w^{(n)}\| \|\varphi\| \ \text{ for all } \varphi \in C_0(\Omega),
\]
we obtain $\|(C_{w,f})^n\| \leq \|w^{(n)}\|$.

On the other hand, given $\eps > 0$, there exists $x_n \in \Omega$ such that $|w^{(n)}(x_n)| \geq \|w^{(n)}\| - \eps$.
Choose an open neighborhood $V_n$ of $f^n(x_n)$ in $\Omega$ such that $\overline{V_n}$ is compact.
By Uryshon's lemma, there exists a continuous map $\phi_n : \Omega \to [0,1]$ such that 
$\supp \phi_n \subset V_n$ and $\phi_n(f^n(x_n)) = 1$.
Hence, $\phi_n \in C_0(\Omega)$ and $\|\phi_n\| = 1$. Since
\[
\|(C_{w,f})^n(\phi_n)\| \geq |\phi_n(f^n(x_n))\, w^{(n)}(x_n)| = |w^{(n)}(x_n)| \geq \|w^{(n)}\| - \eps,
\]
we obtain $\|(C_{w,f})^n\| \geq \|w^{(n)}\| - \eps$. Since $\eps > 0$ is arbitrary, the proof of (\ref{PB-C0X-1}) is complete.

The last assertion follows from (\ref{PB-C0X-1}).
\end{proof}

For weighted shifts $B_w$ on $c_0(\N)$ and $c_0(\Z)$, 
the norms of the iterates of $B_w$ are also given by formulas (\ref{WS-1}) and (\ref{WS-3}), respectively.
Hence, the conditions for $B_w$ to be power-bounded are also given by (\ref{WS-2}) and (\ref{WS-4}), respectively.
Thus, \cite[Corollary 4]{BV} (see also \cite[Proposition~27]{BBMP11}) and \cite[Corollary 5]{BV} can be rewritten as follows.

\begin{corollary}
A weighted shift $B_w$ on $c_0(\N)$ is Li-Yorke chaotic if and only if it is not power-bounded.
\end{corollary}

\begin{corollary}
A weighted shift $B_w$ on $c_0(\Z)$ with nonzero weights is Li-Yorke chaotic if and only if it is not power-bounded 
and $\displaystyle \liminf_{n \to \infty} |w_{-n} \cdots w_{-1}| = 0$.
\end{corollary}

\begin{theorem}\label{LY-PB-C0X}
Given a weighted composition operator $C_{w,f}$ on $C_0(\Omega)$, 
let $\fX$ be the closed subspace of $C_0(\Omega)$ generated by the functions $\varphi \in C_0(\Omega)$ 
whose support is contained in a relatively compact open set $B$ in $\Omega$ satisfying
\begin{equation}\label{LY-PB-C0X-1}
\liminf_{n \to \infty} \|w^{(n)}\|_{f^{-n}(B)} = 0.
\end{equation}
Then, $C_{w,f}$ is Li-Yorke chaotic if and only if it is not power-bounded in $\fX$.
\end{theorem}

\begin{proof}
Suppose that $C_{w,f}$ is Li-Yorke chaotic. By \cite[Theorem~1]{BV} and \cite[Remark~3]{BV}, 
there exist a sequence $(B_i)_{i \in \N}$ of nonempty relatively compact open sets in $\Omega$ and 
an increasing sequence $(n_j)_{j \in \N}$ of positive integers such that $\overline{B_i} \subset B_{i+1}$ for all $i \in \N$,
\begin{equation}\label{LY-PB-C0X-2}
\sup\big\{\|w^{(n)}\|_{f^{-n}(B_i)} : i, n \in \N\big\} = \infty
\end{equation}
and
\begin{equation}\label{LY-PB-C0X-3}
\lim_{j \to \infty} \|w^{(n_j)}\|_{f^{-n_j}(B_i)} = 0 \ \text{ for all } i \in \N.
\end{equation}
Since $\overline{B_i} \subset B_{i+1}$, there is a continuous map $\phi_i : \Omega \to [0,1]$ such that 
$\supp \phi_i \subset B_{i+1}$ and $\phi_i = 1$ on $\overline{B_i}$.
By (\ref{LY-PB-C0X-3}), $\phi_i \in \fX$ for all $i \in \N$. Since $\|\phi_i\| = 1$ and 
\[
\|(C_{w,f})^n(\phi_i)\| \geq \|(\phi_i \circ f^n) \cdot w^{(n)}\|_{f^{-n}(B_i)} = \|w^{(n)}\|_{f^{-n}(B_i)} \ \ \ (i, n \in \N),
\]
it follows from (\ref{LY-PB-C0X-2}) that $C_{w,f}$ is not power-bounded in $\fX$.

Conversely, suppose that $C_{w,f}$ is not Li-Yorke chaotic. 
If $\varphi \in C_0(\Omega)$ has support contained in a relatively compact open set $B$ in $\Omega$ satisfying (\ref{LY-PB-C0X-1}), then
\[
\liminf_{n \to \infty} \|(C_{w,f})^n(\varphi)\| = \liminf_{n \to \infty} \|(\varphi \circ f^n) \cdot w^{(n)}\|_{f^{-n}(B)}
  \leq \|\varphi\| \liminf_{n \to \infty} \|w^{(n)}\|_{f^{-n}(B)} = 0.
\]
Since $C_{w,f}$ does not admit a semi-irregular vector, we conclude that
\[
\lim_{n \to \infty} (C_{w,f})^n(\varphi) = 0.
\]
This implies that the set $\cR_1$ of all $\varphi \in \fX$ whose orbit under $C_{w,f}$ has a subsequence converging to zero 
is dense, hence residual, in $\fX$.
Now, simply continue arguing as at the end of the proof of Theorem~\ref{LY-PB} to conclude that $C_{w,f}$ is power-bounded in $\fX$.
\end{proof}

\begin{theorem}\label{DLY-PB-C0X}
Consider a weighted composition operator $C_{w,f}$ on $C_0(\Omega)$. If 
\[
\liminf_{n \to \infty} \|w^{(n)}\|_{f^{-n}(B)} = 0
\]
for every relatively compact open set $B$ in $\Omega$, then the following assertions are equivalent:
\begin{itemize}
\item [\rm (i)] $C_{w,f}$ is Li-Yorke chaotic;
\item [\rm (ii)] $C_{w,f}$ has a residual set of irregular vectors;
\item [\rm (iii)] $C_{w,f}$ is not power-bounded.
\end{itemize}
\end{theorem}

\begin {proof} 
Given $\varphi \in C_c(\Omega)$, take a relatively compact open set $B$ in $\Omega$ with $\supp \varphi \subset B$. 
By the hypothesis,
\[
\liminf_{n \to \infty} \|(C_{w,f})^n(\varphi)\| \leq \|\varphi\| \liminf_{n \to \infty} \|w^{(n)}\|_{f^{-n}(B)} = 0.
\]
Since $C_c(\Omega)$ is dense in $C_0(\Omega)$, it follows that the set $\cR_1$ of all $\varphi \in C_0(\Omega)$ 
whose orbit under $C_{w,f}$ has a subsequence converging to zero is residual in $C_0(\Omega)$.
If (iii) holds, then the set $\cR_2$ of all $\varphi \in C_0(\Omega)$ 
whose orbit under $C_{w,f}$ is unbounded is also residual in $C_0(\Omega)$, which implies~(ii).
The other implications are clear.
\end{proof}

\begin{remark}\label{R-DLY-PB-C0X}
{\bf (a)} In the case $C_0(\Omega)$ is separable, \cite[Theorem~10]{BBMP2} says that (ii) is equivalent~to
\begin{itemize}
\item [(ii')] $C_{w,f}$ is densely Li-Yorke chaotic.
\end{itemize}

\noindent
{\bf (b)} If the space $C_0(\Omega)$ is separable and there is an increasing sequence $(n_j)_{j \in \N}$ of positive integers such that
\[
\lim_{j \to \infty} \|w^{(n_j)}\|_{f^{-n_j}(B)} = 0
\]
for every relatively compact open set $B$ in $\Omega$, then \cite[Theorem~31]{BBMP2} implies that (i)--(iii) are also equivalent to
\begin{itemize}
\item [(iv)] $C_{w,f}$ admits a dense irregular manifold.
\end{itemize}
\end{remark}

Let us now present some applications to weighted translation operators.

\begin{corollary}
Let $\Omega = [1,\infty)$ or $\R$. For any weighted translation operator $T_w$ on $C_0(\Omega)$,
\[
\|(T_w)^n\| = \sup_{x \in \Omega} |w(x) \cdots w(x+n-1)| \ \text{ for all } n \in \N.
\]
In particular, $T_w$ is power-bounded if and only if 
\[
\sup\{|w(x) \cdots w(x+n-1)| : x \in \Omega, n \in \N\} < \infty.
\]
\end{corollary}

\begin{proof}
This is a particular case of Theorem~\ref{PB-C0X}.
\end{proof}

\begin{corollary}
For weighted shifts $B_w$ on $c_0(\N)$ and weighted translation operators $T_w$ on $C_0[1,\infty)$,
properties (i)--(iv) above are always equivalent to each other.
\end{corollary}

\begin{proof}
The conditions in Remark~\ref{R-DLY-PB-C0X}(b) are satisfied with $(n_j)_{j \in \N}$ being the full sequence $(n)_{n \in \N}$.
\end{proof}

\begin{corollary}
Given a weighted translation operator $T_w$ on $C_0(\R)$, let $\fX$ be the closed subspace of $C_0(\R)$ generated by the 
functions $\varphi \in C_0(\R)$ whose support is contained in a bounded open set $B$ in $\R$ satisfying
\[
\liminf_{n \to \infty} \Big(\sup_{x \in B} |w(x-n) \cdots w(x-1)|\Big) = 0.
\]
Then, the following assertions are equivalent:
\begin{itemize}
\item [\rm (i)] $T_w$ is Li-Yorke chaotic;
\item [\rm (ii)] There exists a sequence $(B_i)_{i \in \N}$ of bounded open sets in $\R$ such that
  \[
  \ov{B_i} \subset B_{i+1} \ \text{ for all } i \in \N,
  \]
  \[
  \liminf_{n \to \infty} \Big(\sup_{x \in B_i} |w(x-n) \cdots w(x-1)|\Big) = 0 \ \text{ for all } i \in \N
  \]
  and
  \[
  \sup\Big\{\sup_{x \in B_i} |w(x-n) \cdots w(x-1)| : i, n \in \N\Big\} = \infty.
  \]
\item [\rm (iii)] $T_w$ is not power-bounded in $\fX$.
\end{itemize}
\end{corollary}

\begin{proof}
The equivalence (i) $\Leftrightarrow$ (ii) follows from \cite[Theorem~1]{BV} and \cite[Remark~3]{BV}.
The equivalence (i) $\Leftrightarrow$ (iii) is a restatement of Theorem~\ref{LY-PB-C0X} in the present particular situation.
\end{proof}

%%%%%%%%%%%%%%%%%%%%%%%%%%%%%%%%%%%%%%%%%%%%%%%%%%%%%%%%%%%%%%%%%%%%%%
%%%%%%%%%%%%%%%%%%%%%%%%%%%%%%%%%%%%%%%%%%%%%%%%%%%%%%%%%%%%%%%%%%%%%%

\section{Distributionally chaotic weighted composition \\ operators}
 
Recall that the {\em lower density} and the {\em upper density} of a set $D \subset \N$ are defined by
\[
\ldens(D) = \liminf_{n \to \infty} \frac{\card(D \cap [1,n])}{n} \ \ \text{ and } \ \
\udens(D) = \limsup_{n \to \infty} \frac{\card(D \cap [1,n])}{n},
\]
respectively. 

Given a metric space $M$, recall that a map $f : M \to M$ is said to be {\em distributionally chaotic} if there exist
an uncountable set $\Gamma \subset M$ and $\eps > 0$ such that each pair $(x,y)$ of distinct points in $\Gamma$
is a {\em distributionally chaotic pair for $f$}, in the sense that
\[
\ldens\{n \in \N : d(f^n(x),f^n(y)) < \eps\} = 0
\]
and
\[
\udens\{n \in \N : d(f^n(x),f^n(y)) < \tau\} = 1 \ \text{ for all } \tau > 0.
\]
If the set $\Gamma$ can be chosen to be dense in $M$, then $f$ is {\em densely distributionally chaotic}.

An extensive study of the concept of distributional chaos in the setting of linear dynamics was developed in \cite{BBMP11,BBMP,BBPW}.
In particular, the following useful characterizations were obtained: 
For any operator $T$ on any Banach space $Y$, the following assertions are equivalent:
\begin{itemize}
\item [(i)] $T$ is distributionally chaotic;
\item [(ii)] $T$ admits a {\em distributionally irregular vector}, that is, 
  a vector $y \in Y$ for which there exist $D,E \subset \N$ with $\udens(D) = \udens(E) = 1$ such that
  \[
   \lim_{n \in D} T^{n}y = 0 \ \ \text{ and } \ \ \lim_{n \in E} \|T^{n}y\| = \infty;
  \]
\item [(iii)] $T$ satisfies the {\em Distributional Chaos Criterion}, that is, 
  there exist sequences $(x_k)_{k \in \N}$ and $(y_k)_{k \in \N}$ in $Y$ such that:
  \begin{itemize}
  \item [(a)] There exists $D \subset \N$ with $\udens(D) = 1$ such that $\lim_{n \in D} T^n x_k = 0$ for all $k$.
  \item [(b)] $y_k \in \ov{\spa\{x_n : n \in \N\}}$ for all $k \in \N$, $\|y_k\| \to 0$ as $k \to \infty$, 
    and there exist $\eps > 0$ and an increasing sequence $(N_k)_{k \in \N}$ of positive integers such that
    \[
    \card\{1 \leq n \leq N_k : \|T^n y_k\| > \eps\} \geq \eps N_k \ \text{ for all } k \in \N.
    \]
    \end{itemize}
  \end{itemize}

Let us also recall that by a {\em distributionally irregular manifold for $T$} we mean a vector subspace of $Y$
consisting, except for the zero vector, of distributionally irregular vectors for $T$.
  
%%%%%%%%%%%%%%%%%%%%%%%%%%%%%%
%%%%%%%%%%%%%%%%%%%%%%%%%%%%%%

\subsection{The case of the space $L^p(\mu)$}
 
\begin{theorem}\label{Dist-1}
A weighted composition operator $C_{w,f}$ on $L^p(\mu)$ is distributionally chaotic if and only if there exists a nonempty countable
family $(B_i)_{i \in I}$ of measurable sets of finite positive $\mu$-measure such that the following properties hold:
\begin{itemize}
\item [\rm (a)] There exists a set $D \subset \N$ with $\udens(D) = 1$ such that
  \[
  \lim_{n \in D} \mu_n(f^{-n}(B_i)) = 0 \ \text{ for all } i \in I.
  \]
  \item [\rm (b)] There exist $\eps > 0$ and an increasing sequence $(N_k)_{k \in \N}$ of positive integers such that,
  for each $k \in \N$, there are $r \in \N$, $i_1,\ldots,i_r \in I$ and $b_1,\ldots,b_r \in (0,\infty)$ with
  \[
  \card\Big\{1 \leq n \leq N_k : \frac{b_1 \mu_n(f^{-n}(B_{i_1})) + \cdots + b_r \mu_n(f^{-n}(B_{i_r}))}
  {b_1 \mu(B_{i_1}) + \cdots + b_r \mu(B_{i_r})} > k \Big\} \geq \eps N_k.
  \]
\end{itemize}
\end{theorem}
 
\begin{proof}
($\Rightarrow$): Let $\varphi \in L^p(\mu)$ be a distributionally irregular vector for $C_{w,f}$.
There exist $D,E \subset \N$ with $\udens(D) = \udens(E) = 1$ such that
\[
\lim_{n \in D} \|(C_{w,f})^n(\varphi)\|_p = 0 \ \ \text{ and } \ \ \lim_{n \in E} \|(C_{w,f})^n(\varphi)\|_p = \infty.
\]
Consider the measurable sets
\[
B_i = \{x \in X : 2^{i-1} \leq |\varphi(x)| < 2^i\} \ \ \ \ \ (i \in \Z)
\]
and let $I = \{i \in \Z : \mu(B_i) > 0\}$. Since
\[
2^{(i-1)p} \mu_n(f^{-n}(B_i)) \leq \|(C_{w,f})^n(\varphi)\|_p^p\,,
\]
we have that property (a) holds. On the other hand, for each $k \in \N$, since
\[
\udens\{n \in \N : \|(C_{w,f})^n(\varphi)\|_p^p > 2^p k \|\varphi\|_p^p\} \geq \udens(E) = 1,
\]
there exists $N_k \in \N$ such that
\[
\card\{1 \leq n \leq N_k : \|(C_{w,f})^n(\varphi)\|_p^p > 2^p k \|\varphi\|_p^p\} \geq N_k \Big(1 - \frac{1}{2k}\Big).
\]
Moreover, it is clear that the $N_k$'s can be chosen so that the sequence $(N_k)_{k \in \N}$ is increasing.
Fix $k \in \N$ and let
\[
J = \{1 \leq n \leq N_k : \|(C_{w,f})^n(\varphi)\|_p^p > 2^p k \|\varphi\|_p^p\}.
\]
Since
\[
2^p k \|\varphi\|_p^p < \|(C_{w,f})^n(\varphi)\|_p^p \leq \sum_{i \in \Z} 2^{ip} \mu_n(f^{-n}(B_i)) \ \text{ for all } n \in J,
\]
we can take $j \in \N$ large enough so that
\[
\sum_{i=-j}^j 2^{ip} \mu_n(f^{-n}(B_i)) > 2^p k \|\varphi\|_p^p \geq k \sum_{i=-j}^j 2^{ip} \mu(B_i) \ \text{ for all } n \in J.
\]
This shows that property (b) holds with $\eps = 1/2$.

\smallskip\noindent
($\Leftarrow$): By property (b), for each $k \in \N$, there are $r_k \in \N$, $i_{k,1},\ldots,i_{k,r_k} \in I$ and 
$b_{k,1},\ldots,b_{k,r_k} \in (0,\infty)$ such that
\begin{equation}\label{EqDist-1}
\card\Big\{1 \leq n \leq N_k : \frac{b_{k,1} \mu_n(f^{-n}(B_{i_{k,1}})) + \cdots + b_{k,r_k} \mu_n(f^{-n}(B_{i_{k,r_k}}))}
{b_{k,1} \mu(B_{i_{k,1}}) + \cdots + b_{k,r_k} \mu(B_{i_{k,r_k}})} > k \Big\} \geq \eps N_k.
\end{equation}
For each $k \in \N$, take pairwise disjoint measurable sets $A_{k,1},\ldots,A_{k,s_k}$ such that
\[
A_{k,1} \cup \ldots \cup A_{k,s_k} = B_{i_{k,1}} \cup \ldots \cup B_{i_{k,r_k}}
\]
and each $B_{i_{k,j}}$ is a union of some of the $A_{k,\ell}$'s.
It follows from property (a) that
\[
\lim_{n \in D}\, (C_{w,f})^n(\rchi_{A_{k,\ell}}) = 0 \ \text{ for all } k \in \N \text{ and } \ell \in \{1,\ldots,s_k\}.
\]
For each $k \in \N$ and each $\ell \in \{1,\ldots,s_k\}$, let $a_{k,\ell} = \sum_j b_{k,j}$,
where the sum is taken over all $j \in \{1,\ldots,r_k\}$ satisfying $A_{k,\ell} \subset B_{i_{k,j}}$.
Consider the measurable simple functions
\[
s_k = \frac{(a_{k,1})^\frac{1}{p} \rchi_{A_{k,1}} + \cdots + (a_{k,s_k})^\frac{1}{p} \rchi_{A_{k,s_k}}}
  {k^\frac{1}{p} \big(a_{k,1} \mu(A_{k,1}) + \cdots + a_{k,s_k} \mu(A_{k,s_k})\big)^\frac{1}{p}} \ \ \ \ \ (k \in \N).
\]
Clearly, $\|s_k\|_p \to 0$ as $k \to \infty$. Moreover, since
\begin{align*}
\|(C_{w,f})^n(s_k)\|_p^p 
  &= \frac{a_{k,1} \mu_n(f^{-n}(A_{k,1})) + \cdots + a_{k,s_k} \mu_n(f^{-n}(A_{k,s_k}))}
     {k\, \big(a_{k,1} \mu(A_{k,1}) + \cdots + a_{k,s_k} \mu(A_{k,s_k})\big)}\\
  &= \frac{b_{k,1} \mu_n(f^{-n}(B_{i_{k,1}})) + \cdots + b_{k,r_k} \mu_n(f^{-n}(B_{i_{k,r_k}}))}
     {k\, \big(b_{k,1} \mu(B_{i_{k,1}}) + \cdots + b_{k,r_k} \mu(B_{i_{k,r_k}})\big)}\,,
\end{align*}
it follows from (\ref{EqDist-1}) that
\[
\card\{1 \leq n \leq N_k : \|(C_{w,f})^n(s_k)\|_p > 1\} \geq \eps N_k.
\]
This shows that the operator $C_{w,f}$ satisfies the Distributional Chaos Criterion, and so it is distributionally chaotic.
\end{proof}

\begin{remark}\label{Disjoint}
Note that the countable family $(B_i)_{i \in I}$ constructed in the proof of Theorem~\ref{Dist-1} 
has the additional property that its terms are pairwise disjoint.
\end{remark}

The next corollaries will give a simple necessary condition and a simple sufficient condition for $C_{w,f}$ to be distributionally chaotic. 
 
\begin{corollary}\label{distr+}
If a weighted composition operator $C_{w,f}$ on $L^p(\mu)$ is distributionally chaotic, then there exists a nonempty countable family 
$(B_i)_{i \in I}$ of pairwise disjoint measurables sets of finite positive $\mu$-measure such that
\begin{itemize}
\item there exists $D \subset \N$ with $\udens(D) = 1$ and $\displaystyle \lim_{n \in D} \mu_n(f^{-n}(B_i)) = 0$ for all $i \in I$,
\item for every constant $C > 0$, $\displaystyle \udens\Big\{n \in \N  : \sup_{i \in I} \frac{\mu_n(f^{-n}(B_i))}{\mu(B_i)} \geq C\Big\} > 0$.
\end{itemize}
\end{corollary}

\begin{proof}
Let $(B_i)_{i \in I}$ be the family given by Theorem~\ref{Dist-1}. 
We may assume that the $B_i$'s are pairwise disjoint (Remark~\ref{Disjoint}).
For each $k \in \N$, let $J_k \subset \{1,\ldots,N_k\}$ be the set that appears in (b). Let $J = \bigcup_{k=1}^\infty J_k$.
Since $\card J_k \geq \eps N_k$ for all $k \in \N$, we have that
\[
\udens(J) \geq \eps.
\]
Fix $C > 0$ and define
\[
H = \Big\{n \in \N  : \sup_{i \in I} \frac{\mu_n(f^{-n}(B_i))}{\mu(B_i)} \geq C\Big\}.
\]
If $n \in J_k \backslash H$ and we use the notation in (b), then
\[
k < \frac{b_1 \mu_n(f^{-n}(B_{i_1})) + \cdots + b_r \mu_n(f^{-n}(B_{i_r}))}{b_1 \mu(B_{i_1}) + \cdots + b_r \mu(B_{i_r})} < C.
\]
Thus, $J \backslash H$ is finite, and so $\udens(H) \geq \udens(J) \geq \eps > 0$.
\end{proof}

\begin{corollary}
Consider a weighted composition operator $C_{w,f}$ on $L^p(\mu)$. If there is a nonempty countable
family $(B_i)_{i \in I}$ of measurable sets of finite positive $\mu$-measure such that
\begin{itemize}
\item there exists $D \subset \N$ with $\udens(D) = 1$ and $\displaystyle \lim_{n \in D} \mu_n(f^{-n}(B_i)) = 0$ for all $i \in I$,
\item there exist $\eps > 0$ and an increasing sequence $(N_k)_{k \in \N}$ of positive integers such that,
  for each $k \in \N$, there exists $i \in I$ with
  \[
  \card\Big\{1 \leq n \leq N_k : \frac{\mu_n(f^{-n}(B_i))}{\mu(B_i)} > k \Big\} \geq \eps N_k,
  \]
\end{itemize}
then $C_{w,f}$ is distributionally chaotic.
\end{corollary}

\begin{proof}
It follows immediately from Theorem~\ref{Dist-1}.
\end{proof}

The fact that the weighted shift on $\ell^p(\N)$ considered in the next example is distributionally chaotic 
is due to Fr\'ed\'eric Bayart and appeared in \cite[Theorem~27]{BBP}. 
In order to illustrate Theorem~\ref{Dist-1} in a concrete situation, we will give below a different proof of this fact.

\begin {example} 
The weighted shift $B_w$ on $\ell^p(\N)$ with weights $w_n = \big(\frac{n+1}{n}\big)^\frac{1}{p}$ is distributionally chaotic.
\end{example}

\begin{proof}
Regard $B_w$ as $C_{w,f}$ by considering $X =\N$, $\fM = \cP(X)$, $\mu$ the counting measure on $\fM$ 
and $f : n \in \N \to n+1 \in \N$. Let $(n_k)_{k \in \N}$ be an increasing sequence of positive integers such that
$\ln n_k > k \ln k$ for every $k \in \N$. 
We define $I$ as the set of all integers $n$ with $n_k \leq n \leq k\, n_k$ for some $k \in \N$ and put $B_i = \{i\}$ for each $i \in I$.
Since condition (a) in Theorem~\ref{Dist-1} holds trivially, let us establish condition (b).
Let $N_k = k\, n_k$ for each $k \in \N$. 
Given $k \geq 2$, let $r = (k-1) n_k$, $i_j = n_k + j$ and $b_j = \frac {1}{i_j}$ ($j \in \{1,\ldots,r\}$).
Then
\[
b_1 \mu(B_{i_1}) + \cdots + b_r \mu(B_{i_r}) = b_1 + \cdots + b_r = \sum_{j=n_k+1}^{k\,n_k} \frac{1}{j} \leq \ln(k\,n_k) - \ln n_k = \ln k.
\]
Moreover, for every $n \in \{n_k,\ldots,(k-1)n_k\}$,
\[
b_1 \mu_n(f^{-n}(B_{i_1})) + \cdots + b_r \mu_n(f^{-n}(B_{i_r})) \geq \sum_{j=1} ^{n_k} \frac{1}{j} \geq \ln n_k.
\]
Thus,
\[
\card\Big\{1 \leq n \leq N_k : 
  \frac{b_1 \mu_n(f^{-n}(B_{i_1})) + \cdots + b_r \mu_n(f^{-n}(B_{i_r}))}{b_1 \mu(B_{i_1}) + \cdots + b_r \mu(B_{i_r})} > k\Big\}
  \geq \Big(\frac{k-2}{k}\Big) N_k.
\]
Hence, by Theorem~\ref{Dist-1}, $B_w$ is distributinally chaotic.
\end{proof}

For general weighted shifts on $\ell^p(\N)$, Theorem~\ref{Dist-1} gives us the following characterization.

\begin{corollary}
A weighted shift $B_w$ on $\ell^p(\N)$ is distributionally chaotic if and only if 
there exist $\eps > 0$ and an increasing sequence $(N_k)_{k \in \N}$ of positive integers such that,
for each $k \in \N$, there are $r \in \N$, $i_1,\ldots,i_r \in \N$ and $b_1,\ldots,b_r \in (0,\infty)$ with
\[
\card\Big\{1 \leq n \leq N_k : \frac{1}{b_1 + \cdots + b_r} \sum_{1 \leq j \leq r, i_j > n} b_j |w_{i_j-n} \cdots w_{i_j-1}|^p > k \Big\} 
  \geq \eps N_k.
\]
\end{corollary}

\begin{proof}
This corollary is essentially a restatement of Theorem~\ref{Dist-1} in the present particular situation.
For the sufficience of the condition, it is enough to consider $B_i = \{i\}$ for each $i \in \N$. Since
\[
\frac{b_1 \mu_n(f^{-n}(B_{i_1})) + \cdots + b_r \mu_n(f^{-n}(B_{i_r}))}{b_1 \mu(B_{i_1}) + \cdots + b_r \mu(B_{i_r})}
  = \frac{1}{b_1 + \cdots + b_r} \sum_{1 \leq j \leq r, i_j > n} b_j |w_{i_j-n} \cdots w_{i_j-1}|^p,
\]
the hypothesis implies that condition (b) in Theorem~\ref{Dist-1} holds; therefore, $B_w$ is distributionally chaotic.
For the converse, it is enough to decompose each set $B_i$ given by Theorem~\ref{Dist-1} as a union of singletons.
\end{proof}

In a similar way, Theorem~\ref{Dist-1} implies the following result for weighted shifts on $\ell^p(\Z)$.

\begin{corollary}
A weighted shift $B_w$ on $\ell^p(\Z)$ is distributionally chaotic if and only if 
there exists a set $S \subset \Z$ such that the following properties hold:
\begin{itemize}
\item [\rm (a)] There exists $D \subset \N$ with $\udens(D) = 1$ and 
  $\displaystyle \lim_{n \in D} |w_{i-n} \cdots w_{i-1}| = 0$ for all $i \in S$.
\item [\rm (b)] There exist $\eps > 0$ and an increasing sequence $(N_k)_{k \in \N}$ of positive integers such that,
  for each $k \in \N$, there are $r \in \N$, $i_1,\ldots,i_r \in S$ and $b_1,\ldots,b_r \in (0,\infty)$ with
  \[
  \card\Big\{1 \leq n \leq N_k : \frac{b_1 |w_{i_1-n} \cdots w_{i_1-1}|^p + \cdots + b_r |w_{i_r-n} \cdots w_{i_r-1}|^p}
  {b_1 + \cdots + b_r} > k \Big\} \geq \eps N_k.
  \]
\end{itemize}  
\end{corollary}

\begin{remark}
If the weights $w_n$ are nonzero and 
\[
\lim_{n \to \infty} \prod_{j=-n+1}^{0} w_j = 0,
\]
then condition (a) above holds for $S = \Z$ and $D = \N$.
Hence, in this case, the weighted shift $B_w$ on $\ell^p(\Z)$ is distributionally chaotic if and only if condition (b) above holds for $S = \Z$.
\end{remark}

Below we will obtain some additional sufficient conditions for $C_{w,f}$ to be distributionally chaotic.

\begin{theorem} Consider a weighted composition operator $C_{w,f}$ on $L^p(\mu)$. 
Suppose that there is a sequence $(A_k)_{k \in \N}$ of measurable sets of finite $\mu$-measure such that the following properties hold: 
\begin{itemize}
\item [\rm (a)] There exists $D \subset \N$ with $\udens(D) = 1$ such that, for each $k \in \N$ and each $\delta > 0$,
  there is a measurable set $B \subset A_k$ with 
  \[
  \mu (A_k \backslash B) < \delta \ \ \text{ and } \ \ \lim_{n \in D} \mu_{n}(f^{-n}(B)) = 0.
  \]
\item [\rm (b)] $\displaystyle \lim_{k \to \infty} \mu(A_k) = 0$ and there exist $\eps > 0$ and an increasing sequence 
  $(N_k)_{k \in \N}$ of positive integers such that
  \[
  \card\{1 \leq n \leq N_k : \mu_n(f^{-n}(A_k)) > \eps\} \geq \eps N_k \ \text{ for all } k \in \N.
  \]
\end{itemize}
Then, $C_{w,f}$ is distributionally chaotic.
\end{theorem}

\begin{proof}
By property (a), for every $k,r \in \N$, there is a measurable set $B_{k,r} \subset A_k$ such that 
  \[
  \mu (A_k \backslash B_{k,r}) < \frac{1}{r} \ \ \text{ and } \ \ \lim_{n \in D} \mu_{n}(f^{-n}(B_{k,r})) = 0.
  \]
Hence, $\rchi_{A_k} \in \ov{\spa\{\rchi_{B_{j,r}} : j,r \in \N\}}$ for all $k \in \N$, and
\[
\lim_{n \in D} (C_{w,f})^n(\rchi_{B_{k,r}}) = 0 \ \text{ for all } k,r \in \N.
\]
Moreover, by property (b), $\|\rchi_{A_k}\|_p \to 0$ as $k \to \infty$, and 
\[
\card\{1 \leq n \leq N_k : \|(C_{w,f})^n(\rchi_{A_k})\|_p > \eps^\frac{1}{p}\} \geq \eps N_k \ \text{ for all } k \in \N. 
\]
This shows that $C_{w,f}$ satisfies the Distributional Chaos Criterion.
\end{proof}

\begin{theorem}\label{DCSum}
Consider a weighted composition operator $C_{w,f}$ on $L^p(\mu)$ with positive weight function $w : X \to (0,\infty)$.
Suppose that there is a sequence $(A_k)_{k \in \N}$ of measurable sets of finite positive $\mu$-measure such that 
the following properties hold:
\begin{itemize}
\item [\rm (a)] There exists a set $D \subset \N$ with $\udens(D) = 1$ such that 
  \[
  \lim_{n \in D} \mu_{n}(f^{-n}(A_k)) = 0 \ \text{ for all } k \in \N.
  \]

\item [\rm (b)] There exists a set $E \subset \N$ with $\udens(E) > 0$ such that 
  \[
  \sum_{n \in E} \left(\frac{\mu(A_n)}{\mu_n(f^{-n}(A_n))}\right)^{\frac{1}{p}} < \infty.
  \]
\end{itemize}
Then, $C_{w,f}$ is distributionally chaotic.
\end{theorem}

\begin{proof}
We will apply the Distributional Chaos Criterion.
For condition (a) in the criterion, we define $x_k = \rchi_{A_k}$ and observe that
\[
\lim_{n \in D} \|(C_{w,f})^n(x_k)\|_p = \lim_{n \in D} \mu_{n}(f^{-n}(A_k))^{\frac{1}{p}} = 0 \ \text{ for all } k \in \N, 
\]
because of hypothesis (a). By using hypothesis (b), we will prove the existence of a vector 
\[
y \in \ov{\spa\{x_k : k \in \N\}}
\]
such that
\[
\lim_{n \in E} \|(C_{w,f})^n(y)\|_p = \infty.
\]
In view of \cite[Proposition~8]{BBMP}, this will imply that condition (b) in the criterion also holds, which will complete the proof.
First, let us prove that hypothesis (b) implies the existence of a sequence $(c_n)_{n \in \N}$ of non-negative scalars such that 
\[
\sum_{n \in E} c_n\, \mu(A_n)^{\frac{1}{p}} < \infty \ \ \text{ and } \ \ \lim_{n \in E} c_n\, \mu_{n}(f^{-n}(A_n))^{\frac{1}{p}} = \infty. 
\]
For this purpose, we will use an argument similar to one used in \cite{Chen}. Let 
\begin{align*}
 a_n &=\left\{
  \begin{array}{cl}
    \left(\frac{\mu(A_n)}{\mu_n(f^{-n}(A_n))}\right) ^{\frac{1}{p}}  &  \text{ if }  n \in E\\
    0  &  \text{ if } n \notin E,
  \end{array}
	     \right.\\
r_n &= \displaystyle \sum _{i\ge n} a_i\,,\\
c_n & =\left\{
  \begin{array}{cl}
    \frac{1}{\sqrt{r_n}\, \mu_n(f^{-n}(A_n))^{\frac{1}{p}} }  &  \text{ if }  n \in E\\
    0  &  \text{ if } n \notin E.
  \end{array}
	     \right.\\
\end{align*}
Then,
\[
\lim_{n \in E} c_n\, \mu_{n}(f^{-n}(A_n))^{\frac{1}{p}} = \lim_{n \in E} \frac{1}{\sqrt{r_n}} = \infty,
\]
since $\sum_{n \in E} a_n$ is convergent. Moreover, we have that 
\[
\sum_{n \in \N} \frac{a_n}{\sqrt{r_n}} \leq 2\sqrt{r_1}\,,
\]
since 
\[
\frac{a_n}{\sqrt{r_n}} = \frac{(\sqrt{r_n} + \sqrt{r_{n+1}})(\sqrt{r_n} - \sqrt{r_{n+1}})}{\sqrt{r_n}} \leq 2 (\sqrt{r_n} - \sqrt{r_{n+1}}).
\]
Hence,
\[
\sum_{n \in E} c_n\, \mu(A_n)^{\frac{1}{p}} = \sum_{n \in E} \frac{a_n}{\sqrt{r_n}} = \sum_{n \in \N} \frac{a_n}{\sqrt{r_n}}
  \leq 2 \sqrt{r_1} \leq 2 \sqrt{\sum_{n \in E} \left(\frac{\mu(A_n)}{\mu_n(f^{-n}(A_n))}\right)^{\frac{1}{p}}} < \infty.
\]
Now, by defining $y = \displaystyle \sum_{n \in E} c_n\, \rchi_{A_n}$, we have that
\[
\|y\|_p = \Big\|\sum_{n \in E} c_n\, \rchi_{A_n}\Big\|_p \leq \sum_{n \in E} c_n\, \mu(A_n)^{\frac{1}{p}} < \infty.
\]
Thus, $y \in \ov{\spa\{x_k : k \in \N\}}$. Moreover,
 \begin{align*}
\lim_{n \in E} \|(C_{w,f})^n(y)\|_p &= \lim_{n \in E} \Big\|\sum_{k \in E} c_k\, (C_{w,f})^n(\rchi_{A_k})\Big\|_p\\
  &\geq \lim_{n \in E} \|c_n\, (C_{w,f})^n(\rchi_{A_n})\|_p\\
  &= \lim_{n \in E} c_n\, \mu_{n}(f^{-n}(A_n))^{\frac{1}{p}} = \infty.
\end{align*}
This completes the proof.
\end{proof}

\begin{corollary} 
Consider a weighted composition operator $C_{w,f}$ on $L^p(\mu)$ with positive weight function $w : X \to (0,\infty)$.
If there is a measurable set $A$ of finite positive $\mu$-measure such that 
\begin{itemize}
\item there exists $D \subset \N$ with $\udens(D) = 1$ and $\displaystyle \lim_{n \in D} \mu_{n}(f^{-n}(A)) = 0$, 

\item there exists $E \subset \N$ with $\udens(E) > 0$ and $\displaystyle \sum_{n \in E} \frac{1}{\mu_n(f^{-n}(A))^{\frac{1}{p}}} < \infty$,
\end{itemize}
then $C_{w,f}$ is distributionally chaotic.
\end{corollary}

In the sequel we will present several sufficient conditions for $C_{w,f}$ to be densely distributionally chaotic.
 
\begin{theorem} 
Consider a weighted composition operator $C_{w,f}$ on $L^p(\mu)$.
If the space $L^p(\mu)$ is separable and for every measurable set $A$ of finite $\mu$-measure and for every $\eps > 0$, 
there is a measurable set $B \subset  A$ with 
\[
\mu(A \backslash B) < \eps \ \ \text{ and } \ \ \lim_{n \to \infty} \mu_{n}(f^{-n}(B)) = 0,
\]
then the following assertions are equivalent:
\begin{itemize}
\item [\rm (i)] $C_{w,f}$ is distributionally chaotic;
\item [\rm (ii)] $C_{w,f}$ is densely distributionally chaotic;
\item [\rm (iii)] $C_{w,f}$ admits a dense distributionally irregular manifold;
\item [\rm (iv)] There exist $\phi \in L^p(\mu)$ and $\delta > 0$ such that 
  \[
  \udens\{n \in \N : \|(C_{w,f})^n(\phi)\|_p \geq \delta\} > 0.
  \]
\end{itemize}
\end{theorem}

\begin{proof}
(i) $\Rightarrow$ (iv): Take a distributionally irregular vector $\phi \in L^p(\mu)$ for $C_{w,f}$.

\smallskip\noindent
(iv) $\Rightarrow$ (iii): Let $\fX_0$ be the set of all simple functions of the form $\sum_{j=1}^m b_j \rchi_{B_j}$,
where $b_1,\ldots,b_m$ are scalars and each $B_j$ is a measurable set of finite $\mu$-measure satisfying 
\[
\lim_{n \to \infty} \mu_n(f^{-n}(B_j)) = 0.
\]
By the assumption, $\fX_0 $ is dense in $L^p(\mu)$. Moreover,
\[
\lim_{n \to \infty} (C_{w,f})^n(\varphi) = 0 \ \text{ for all } \varphi \in \fX_0.
\]
Hence, by (iv) and \cite[Theorem~33]{BBPW}, (iii) holds.   

\smallskip\noindent
(iii) $\Rightarrow$ (ii) $\Rightarrow$ (i): Obvious.
\end{proof}

\begin{corollary} 
Assume the hypotheses of the previous theorem.
If there exist a constant $C > 0$ and a measurable set $A$ of finite positive $\mu$-measure such that
\[
\udens\{n \in \N : \mu_n(f^{-n}(A)) \geq C \} > 0,
\]
then $C_{w,f}$ is densely distributionally chaotic.
\end{corollary}

\begin{proof}
Property (iv) in the previous theorem holds with $\phi = \rchi_A$ and $\delta = C^\frac{1}{p}$.
\end{proof}

Recall that a {\em measurable system} is a $4$-tuple $(X,\fM,\mu,f)$ such that: 
\begin{itemize}
\item [($\alpha$)] $(X,\fM,\mu)$ is a $\sigma$-finite measure space with $\mu(X) > 0$; 
\item [($\beta$)] $f : X \to X$ is a non-singular bimeasurable bijective map 
  (where $f$ {\em non-singular} means that $\mu(f^{-1}(B)) = 0$ if and only if $\mu(B) = 0$); 
\item [($\gamma$)] there exists $c > 0$ such that $\mu(f^{-1}(B)) \leq c\, \mu(B)$ for every $B \in \fM$.
\end{itemize}
Note that condition ($\gamma$) implies that the composition operator 
$C_f(\varphi) = \varphi \circ f$ is a well-defined bounded linear operator on $L^p (\mu)$.
Recall also that a measurable system $(X,\fM,\mu,f)$ (equivalently, the map $f$) is said to be:
\begin{itemize}
\item {\em conservative} if for each measurable set $B$ of positive $\mu$-measure, there exists $n \in \N$ such that 
  $\mu(B \cap f^{-n}(B)) > 0$;
\item {\em dissipative} if there is a measurable set $W$ (called a {\em wandering set}) such that the sets $f^n(W)$, $n \in \Z$, 
  are pairwise disjoint and $X = \bigcup_{n \in \Z} f^n(W)$. 
\end{itemize}
Finally, recall that a set $A \subset  X$ is said to be $f$-invariant if $ f^{-1}(A) = A$. 
The following classic result can be found in \cite[Theorem 3.2]{K}:

\medskip\noindent
{\bf Hopf Decomposition Theorem.}  {\it If $(X,\fM,\mu)$ is a $\sigma$-finite measure space and $f : X \to X$ is a non-singular
measurable map, then $X$ can be written as the union of two disjoint $f$-invariant sets $\cC(f)$ and $\cD(f)$, 
called the {\em conservative} and the {\em dissipative parts} of $f$, respectively, such that 
$f|_{\cC(f)}$ is conservative and $f|_{\cD(f)}$ is dissipative.}

\begin{theorem} Consider a measurable system $(X,\fM,\mu,f)$ and assume the following conditions:
\begin{itemize} 
\item [\rm (a)] For every measurable set $A$ of finite $\mu$-measure and for every $\eps > 0$, 
  there is a measurable set $B \subset  A$ with 
  \[
  \mu(A \backslash B) < \eps \ \ \text{ and } \ \ \lim_{n \to \infty} \mu(f^{-n}(B)) = 0.
  \]
\item [\rm (b)] There is a measurable set $B' \subset W \subset \cD(f)$ of finite positive $\mu$-measure, 
  $W$ being a wandering set of $\cD(f)$, such that
  \[
  \sum_{n \in \Z} \mu(f^{n}(B')) \text{ converges}.
  \]
\end{itemize}
Then, the composition operator $C_f$ is densely distributionally chaotic.
\end{theorem}

\begin {proof}
Let $\fX_0$ be the set of all simple functions of the form $\sum_{j=1}^m b_j \rchi_{B_j}$,
where $b_1,\ldots,b_m$ are scalars and each $B_j$ is a measurable set of finite $\mu$-measure satisfying 
\[
\lim_{n \to \infty} \mu(f^{-n}(B_j)) = 0.
\]
By property (a), $\fX_0 $ is dense in $L^p(\mu)$. Moreover, $\lim_{n \to \infty} (C_f)^n(\varphi) = 0$ for all $\varphi \in \fX_0$.
Therefore, by \cite[Theorem~19]{BBMP}, it remains to show that there exist a subset $Y$ of $L^p(\mu)$, 
a map $S : Y \to Y$ with $C_f(S(\varphi)) = \varphi$ on $Y$, and a vector $\phi \in Y \backslash \{0\}$ such that  the series
\[
\sum_{n=1}^{\infty} (C_f)^n(\phi) \ \text{ and } \sum_{n=1}^{\infty} S^n(\phi) \ \text{ converge unconditionally}.
\]
For this purpose, we put 
\[
Y = \Big\{\rchi_B : B \in \fM \text{ and } \sum_{n \in \Z} \mu(f^n(B)) \text{ converges}\Big\}
\]
and $\phi = \rchi_{B'} \in Y \backslash \{0\}$. 
Since $B' \subset W \subset \cD(f)$ and $W$ is a wandering set, we have that the sets $f^n(B')$, $n \in \Z$, are pairwise disjoint.
Hence, the convergence of the series in (b) implies that, for any increasing sequence $(n_j)_{j \in \N}$ of positive integers,
\[
\rchi_{\bigcup_{j = 1}^m f^{-n_j}(B')} \in L^p(\mu) \ \text{ for all } m \in \N
\]
and
\[
\sum_{j = 1}^m (C_f)^{n_j}(\phi) =\sum_{j = 1}^m \rchi_{f^{-n_j}(B')} = \rchi_{\bigcup_{j = 1}^m f^{-n_j}(B')}
  \to \rchi_{\bigcup_{j = 1}^\infty f^{-n_j}(B')} \text{ in } L^p(\mu) \text{ as } m \to \infty.
\]
Hence, the series $\sum_{j = 1}^\infty (C_f)^{n_j}(\phi)$ converges in $L^p(\mu)$.
This implies that the series $\sum_{n=1}^{\infty} (C_f)^n(\phi)$ is unconditionally convergent.
Now, we define $S : Y \to Y$ by
\[
S(\varphi) = \rchi_{f(B)}, \text{ for each } \varphi = \rchi_{B} \in Y.
\]
Clearly, $C_f(S(\varphi)) = \varphi$. By arguing as above, it follows that the series
\[
\sum_{n=1}^\infty S^n(\phi) = \sum_{n=1}^\infty \rchi_{f^n(B')} \ \text{ converges unconditionally}.
\]
This completes the proof.
\end{proof}

\begin{corollary} Let $(X,\fM,\mu,f)$ be a non-conservative measurable system with $\mu(X) < \infty$. 
If for every measurable set $A$ of finite $\mu$-measure and for every $\eps > 0$, there is a measurable set $B \subset  A$ with 
$\mu(A \backslash B) < \eps$ and $\lim_{n \to \infty} \mu(f^{-n}(B)) = 0$, the $C_f$ is densely distributionally chaotic.
\end{corollary}

\begin {proof}
Since the measurable system is not conservative, there is a measurable set $B' \subset W \subset \cD(f)$ with $\mu(B') > 0$.
Since
\[
\sum_{n \in \Z} \mu(f^{n}(B')) \leq \mu(X) < \infty,
\]
the result follows from the previous theorem.
\end{proof}

\begin{theorem}\label{ddcwc} 
Consider a weighted composition operator $C_{w,f}$ on $L^p(\mu)$ with positive weight function $w : X \to (0,\infty)$.
Suppose that the measure space $(X,\fM,\mu)$ is $\sigma$-finite, the bimeasurable map $f$ is injective and non-singular,
and the following properties hold: 
\begin{itemize} 
\item [\rm (a)] For every measurable set $A$ of finite $\mu$-measure and for every $\eps > 0$, 
  there is a measurable set $B \subset  A$ with 
  \[
  \mu(A \backslash B) < \eps \ \ \text{ and } \ \ \lim_{n \to \infty} \mu_n(f^{-n}(B)) = 0.
  \]
\item [\rm (b)] There is a measurable set $B' \subset W \subset \cD(f)$ of finite positive $\mu$-measure, 
  $W$ being a wandering set of $\cD(f)$, and a subset $D$ of $\N$ with $\udens(D) > 0$ such that
  \[
  \sum_{n \in D} \int_{f^n(B')} \left(\prod_{j=0}^{n-1} (w \circ f^{j-n})\right)^{-p} d\mu < \infty.
  \]
\end{itemize}
Then, $C_{w,f}$ is densely distributionally chaotic.
\end{theorem}

\begin{proof}
For each $n \in \N$, we define $v_n : X \to (0,\infty)$ by
\[
v_n = \prod_{j=0}^{n-1} (w \circ f^{j-n}) \text{ on } f^n(X) \ \ \ \text{ and } \ \ \ v_n = 1 \text{ on } X \backslash f^n(X).
\]
Note that $v_n$ is well-defined because $f$ is injective. For each $k \in \N$, let
\[
\phi_k = \sum_{n \in D, n \geq k} \frac{1}{v_n}\, \rchi_{f^n(B')}.
\]
The convergence of the series in (b) implies that each  $\phi_k$ belongs to $L^p(\mu)$ and $\phi_k \to 0$ as $k \to \infty$.
Moreover,
\[
(C_{w,f})^n(\phi_k) = (C_{w,f})^n\Big(\frac{1}{v_n}\, \rchi_{f^n(B')}\Big) + \cdots = \rchi_{B'}+ \cdots,
\]
whenever $n \in D$ and $n \geq k$. Hence, by defining $\eps = \mu(B')^{\frac{1}{p}} > 0$, we have that
\[
\|(C_{w,f})^n(\phi_k)\|_p > \eps \ \text{ for all } n \in D \text{ with } n \geq k.
\]
Since $\delta = \udens(D) > 0$, there is an increasing sequence $(N_k)_{k \in \N}$ of positive integers such that
\[
\card\{1 \leq n \leq N_k : \|(C_{w,f})^n(\phi_k)\|_p > \eps\} \geq \frac{\delta}{2}\, N_k \ \text{ for all } k \in \N.
\]
By \cite[Proposition~8]{BBMP}, $C_{w,f}$ admits a distributionally unbounded orbit. 
Thus, by \cite [Theorem~15]{BBMP}, $C_{w,f}$ is densely distributionally chaotic.
\end{proof}

\begin{corollary} 
Consider a weighted shift $B_w$ on $\ell^p(\N)$ with positive weights. 
If there exists a set $D \subset \N$ with $\udens(D) > 0$ such that
\[
\sum_{n \in D} \left(\prod_{j=1}^{n} w_j \right)^{-p} < \infty,
\]
then $B_w$ is densely distributionally chaotic.
\end{corollary}

\begin{proof}
In the present case, $\mu$ is the counting measure on $\N$ and $f : n \in \N \to n+1 \in \N$. 
Condition (a) in Theorem~\ref{ddcwc} is trivially true.
By defining $B' = \{1\}$, we have that $B'$ is a wandering set with finite positive $\mu$-measure such that 
\[
\sum_{n \in D} \int_ {f^n(B')} \left(\prod_{j=0}^{n-1} (w \circ f^{j-n})\right)^{-p} d\mu
  = \sum_{n \in D} \left(\prod_{j=1}^{n} w_j\right)^{-p} < \infty.
\]
Hence, condition (b) in Theorem~\ref{ddcwc} is also true. Thus, $B_w$ is densely distributionally chaotic.
\end{proof}

\begin{corollary}
Consider a weighted shift $B_w$ on $\ell^p(\Z)$ with positive weights. If 
\[
\lim_{n \to \infty} \prod_{j=-n+1}^{0} w_j = 0
\]
and there exists a set $D \subset \N$ with $\udens(D) > 0$ such that
\[
\sum_{n \in D} \left(\prod_{j=1}^{n} w_j \right)^{-p}  < \infty,
\]
then $B_w$ is densely distributionally chaotic.
\end{corollary}

\begin{proof}
In the present case, $\mu$ is the counting measure on $\Z$ and $ f : n \in \Z \to n+1 \in \Z$. 
Since $\lim_{n \to \infty} (w_{-n+1} \cdots w_0) = 0$, for every finite set $A \subset \Z$, 
\[
\lim_{n \to \infty} \mu_n(f^{-n}(A)) = \lim_{n \to \infty} \sum_{i \in A} \mu_n(f^{-n}(\{i\})) 
  = \lim_{n \to \infty} \sum_{i \in A} (w_{-n+i} \cdots w_{-1+i}) = 0.
\]
Now, $B' = \{1\}$ is a wandering set with finite positive $\mu$-measure such that
\[
\sum_{n \in D} \int_{f^{n}(B')} \left(\prod_{j=0}^{n-1} (w \circ f^{j-n})\right)^{-p} d\mu 
  = \sum_{n \in D} \left(\prod_{j=1}^{n} w_j\right)^{-p} < \infty.
\]
Thus, by Theoren \ref{ddcwc}, $B_w$ is densely distributionally chaotic.
\end{proof}

Let us mention that reference \cite{MOP} contains some sufficient conditions for weighted shifts on K\"othe sequence spaces
to be distributionally chaotic.

%%%%%%%%%%%%%%%%%%%%%%%%%%%%%%
%%%%%%%%%%%%%%%%%%%%%%%%%%%%%%

\subsection{The case of the space $C_0(\Omega)$}

\begin{theorem}\label{Dist-C0X-1}
A weighted composition operator $C_{w,f}$ on $C_0(\Omega)$ is distributionally chaotic if and only if there exists a sequence
$(B_i)_{i \in \N}$ of relatively compact open sets in $\Omega$ such that the following properties hold:
\begin{itemize}
\item [\rm (a)] There exists a set $D \subset \N$ with $\udens(D) = 1$ such that
  \[
  \lim_{n \in D} \|w^{(n)}\|_{f^{-n}(B_i)} = 0 \ \text{ for all } i \in \N.
  \]
  \item [\rm (b)] There exist $\eps > 0$ and an increasing sequence $(N_k)_{k \in \N}$ of positive integers such that,
  for each $k \in \N$, there exists $i \in \N$ with
  \[
  \card\{1 \leq n \leq N_k : \|w^{(n)}\|_{f^{-n}(B_i)} > k\} \geq \eps N_k.
  \]
\end{itemize}
\end{theorem}
 
\begin{proof}
($\Rightarrow$): Let $\varphi \in C_0(\Omega)$ be a distributionally irregular vector for $C_{w,f}$.
There exist $D,E \subset \N$ with $\udens(D) = \udens(E) = 1$ such that
\[
\lim_{n \in D} \|(C_{w,f})^n(\varphi)\| = 0 \ \ \text{ and } \ \ \lim_{n \in E} \|(C_{w,f})^n(\varphi)\| = \infty.
\]
Consider the relatively compact open sets
\[
B_i = \{x \in \Omega : |\varphi(x)| > i^{-1}\} \ \ \ \ \ (i \in \N).
\]
Since
\[
i^{-1} \|w^{(n)}\|_{f^{-n}(B_i)} \leq \|(\varphi \circ f^n) \cdot w^{(n)}\|_{f^{-n}(B_i)} \leq \|(C_{w,f})^n(\varphi)\|,
\]
we have that property (a) holds. On the other hand, for each $k \in \N$, since
\[
\udens\{n \in \N : \|(C_{w,f})^n(\varphi)\| > k \|\varphi\|\} \geq \udens(E) = 1,
\]
there exists $N_k \in \N$ such that
\[
\card\{1 \leq n \leq N_k : \|(C_{w,f})^n(\varphi)\| > k \|\varphi\|\} \geq N_k \Big(1 - \frac{1}{2k}\Big).
\]
Moreover, the $N_k$'s can be chosen so that the sequence $(N_k)_{k \in \N}$ is increasing.
Since
\[
\|(C_{w,f})^n(\varphi)\| = \sup_{i \in \N} \|(\varphi \circ f^n) \cdot w^{(n)}\|_{f^{-n}(B_i)} 
  \leq \|\varphi\| \sup_{i \in \N} \|w^{(n)}\|_{f^{-n}(B_i)},
\]
we see that there exists $i_k \in \N$ such that
\[
\card\{1 \leq n \leq N_k : \|w^{(n)}\|_{f^{-n}(B_{i_k})} > k\} \geq N_k \Big(1 - \frac{1}{2k}\Big).
\]
Thus, property (b) holds with $\eps = 1/2$.

\smallskip\noindent
($\Leftarrow$): By property (b), for each $k \in \N$, there exists $i_k \in \N$ such that
\[
\card F_k \geq \eps N_k, \ \text{ where } F_k = \{1 \leq n \leq N_k : \|w^{(n)}\|_{f^{-n}(B_{i_k})} > k\}.
\]
We shall construct a sequence $(A_k)_{k \in \N}$ of relatively compact open sets in $\Omega$ such that
$\ov{A_k} \subset B_{i_1} \cup \ldots \cup B_{i_k}$, $\ov{A_k} \subset A_{k+1}$ and
\[
\{1 \leq n \leq N_k : \|w^{(n)}\|_{f^{-n}(A_k)} > k\} \supset F_k \ \text{ for all } k \in \N.
\]
For this purpose, for each $n \in F_k$, take $x_{k,n} \in f^{-n}(B_{i_k})$ such that $|w^{(n)}(x_{k,n})| > k$.
We begin by choosing an open set $A_1$ in $\Omega$ satisfying
\[
\{f^n(x_{1,n}) : n \in F_1\} \subset A_1 \subset \ov{A_1} \subset B_{i_1}.
\]
If $k \geq 2$ and $A_1,\ldots,A_{k-1}$ have already been chosen, then take an open set $A_k$ in $\Omega$ such that
\[
\ov{A_{k-1}} \cup \{f^n(x_{k,n}) : n \in F_k\} \subset A_k \subset \ov{A_k} \subset B_{i_1} \cup \ldots \cup B_{i_k}.
\]
It is clear that the sequence $(A_k)_{k \in \N}$ constructed in this way has the desired properties.
Now, for each $k \in \N$, take a continuous map $\phi_k : \Omega \to [0,1]$ such that
\[
\phi_k = 1 \text{ on } \ov{A_k} \ \ \text{ and } \ \ \supp \phi_k \subset A_{k + 1}.
\]
It follows from property (a) that
\[
\lim_{n \in D} \|(C_{w,f})^n(\phi_k)\| \leq \lim_{n \in D} \|w^{(n)}\|_{f^{-n}(A_{k + 1})} = 0 \ \text{ for all } k \in \N.
\]
Let $\varphi_k = \frac{1}{k} \phi_k$ for each $k \in \N$. Obviously, $\varphi_k \to 0$ as $k \to \infty$.
If $n \in \{1,\ldots,N_k\}$ satisfies $\|w^{(n)}\|_{f^{-n}(A_k)} > k$, then
\[
\|(C_{w,f})^n(\varphi_k)\| \geq \|(\varphi_k \circ f^n) \cdot w^{(n)}\|_{f^{-n}(A_k)} = \frac{1}{k}\, \|w^{(n)}\|_{f^{-n}(A_k)} > 1.
\]
Thus,
\[
\card\{1 \leq n \leq N_k : \|(C_{w,f})^n(\varphi_k)\| > 1\} \geq \card F_k \geq \eps N_k \ \text{ for all } k \in \N.
\]
Hence, by the Distributional Chaos Criterion, $C_{w,f}$ is distributionally chaotic.
\end{proof}

\begin{remark}
Note that the sequence $(B_i)_{i \in \N}$ constructed in the proof of Theorem~\ref{Dist-C0X-1} has the following additional property:
$\ov{B_i} \subset B_{i+1}$ for all $i \in \N$.
\end{remark}

\begin{corollary}\label{dcuws-c0}
A weighted shift $B_w$ on $c_0(\N)$ is distributionally chaotic if and only if
\begin{equation}\label{eq-dcuws-c0}
\inf_{k \in \N} \Big(\sup_{N \in \N} \frac{\card\{1 \leq n \leq N : |w_i \cdots w_{i+n-1}| > k \text{ for some } i \in \N\}}{N}\Big) > 0.
\end{equation}
\end{corollary}

\begin{proof}
($\Leftarrow$): Since condition (a) in Theorem~\ref{Dist-C0X-1} is superfluous in the present case, it is enough to establish condition~(b).
The hypothesis implies the existence of an $\eps > 0$ and a sequence $(N_k)_{k \in \N}$ of positive integers such that
\[
\card I_k > \eps N_k, \ \text{ where } I_k = \{1 \leq n \leq N_k : |w_i \cdots w_{i+n-1}| > k \text{ for some } i \in \N\}.
\]
By definition, for each $n \in I_k$, there exists $i_{k,n} \in \N$ such that $|w_{i_{k,n}} \cdots w_{i_{k,n}+n-1}| > k$. Let
\[
B_k = \{i_{k,n} + n : n \in I_k\}.
\]
Since $\|w^{(n)}\|_{f^{-n}(B_k)} \geq |w^{(n)}(i_{k,n})| = |w_{i_{k,n}} \cdots w_{i_{k,n}+n-1}| > k$ for every $n \in I_k$, we obtain
\[
\card\{1 \leq n \leq N_k : \|w^{(n)}\|_{f^{-n}(B_k)} > k\} \geq \card I_k > \eps N_k.
\]
Now, by passing to a subsequence, if necessary, we may assume that the sequence $(N_k)_{k \in \N}$ is increasing.
Thus, condition (b) in Theorem~\ref{Dist-C0X-1} holds and $B_w$ is distributionally chaotic.

\smallskip\noindent
($\Rightarrow$): By Theorem~\ref{Dist-C0X-1}, for each $k \in \N$, there exists $i_k \in \N$ such that
\[
\card\{1 \leq n \leq N_k : \|w^{(n)}\|_{f^{-n}(B_{i_k})} > k\} \geq \eps N_k.
\]
If $n \in \{1,\ldots,N_k\}$ satisfies $\|w^{(n)}\|_{f^{-n}(B_{i_k})} > k$, then $|w_i \cdots w_{i+n-1}| = |w^{(n)}(i)| > k$
for some $i \in f^{-n}(B_{i_k})$. Therefore,
\[
\frac{\card\{1 \leq n \leq N_k : |w_i \cdots w_{i+n-1}| > k \text{ for some } i \in \N\}}{N_k} \geq \eps.
\]
This clearly implies (\ref{eq-dcuws-c0}).
\end{proof}

\begin{corollary}\label{dcbws-c0}
A weighted shift $B_w$ on $c_0(\Z)$ is distributionally chaotic if and only if
there exist a set $S \subset \Z$ and a set $D \subset \N$ with $\udens(D) = 1$ such that
\begin{equation}\label{eq1-dcbws-c0}
\lim_{n \in D} |w_{i-n} \cdots w_{i-1}| = 0 \ \text{ for all } i \in S
\end{equation}
and
\begin{equation}\label{eq2-dcbws-c0}
\inf_{k \in \N} \Big(\sup_{N \in \N} \frac{\card\{1 \leq n \leq N : |w_{i-n} \cdots w_{i-1}| > k \text{ for some } i \in S\}}{N}\Big) > 0.
\end{equation}
\end{corollary}

\begin{proof}
($\Leftarrow$): Let $(B_i)_{i \in \N}$ be an enumeration of the collection of all nonempty finite subsets of $S$. For every $i \in \N$,
\[
\lim_{n \in D} \|w^{(n)}\|_{f^{-n}(B_i)} = \lim_{n \in D} \Big(\max_{j \in B_i} |w_{j-n} \cdots w_{j-1}|\Big) = 0,
\]
because of (\ref{eq1-dcbws-c0}). Hence, condition (a) in Theorem~\ref{Dist-C0X-1} holds.
By (\ref{eq2-dcbws-c0}), there exist $\eps > 0$ and a sequence $(N_k)_{k \in \N}$ of positive integers such that
\[
\card I_k > \eps N_k, \ \text{ where } I_k = \{1 \leq n \leq N_k : |w_{i-n} \cdots w_{i-1}| > k \text{ for some } i \in S\}.
\]
For each $n \in I_k$, take $i_{k,n} \in S$ such that $|w_{i_{k,n}-n} \cdots w_{i_{k,n}-1}| > k$. 
Then, $\{i_{k,n} : n \in I_k\} = B_{i_k}$ for some $i_k \in \N$. Since 
\[
\card\{1 \leq n \leq N_k : \|w^{(n)}\|_{f^{-n}(B_{i_k})} > k\} \geq \card I_k > \eps N_k,
\]
we see that (by passing to a subsequence, if necessary) condition (b) in Theorem~\ref{Dist-C0X-1} also holds.
Thus, $B_w$ is distributionally chaotic.

\smallskip\noindent
($\Rightarrow$): Let $(B_i)_{i \in \N}$ be the sequence given by Theorem~\ref{Dist-C0X-1}.
Define $S = \bigcup_{i \in \N} B_i \subset \Z$. 
It is straightforward to check that properties (a) and (b) in Theorem~\ref{Dist-C0X-1} imply 
(\ref{eq1-dcbws-c0}) and (\ref{eq2-dcbws-c0}), respectively.
\end{proof}

\begin{remark}
If the weights $w_n$ are nonzero and 
\[
\lim_{n \to \infty} \prod_{j=-n+1}^{0} w_j = 0,
\]
then (\ref{eq1-dcbws-c0}) holds for $S = \Z$ and $D = \N$.
Hence, in this case, the weighted shift $B_w$ on $c_0(\Z)$ is distributionally chaotic if and only if
\[
\inf_{k \in \N} \Big(\sup_{N \in \N} \frac{\card\{1 \leq n \leq N : |w_i \cdots w_{i+n-1}| > k \text{ for some } i \in \Z\}}{N}\Big) > 0.
\]
\end{remark}

The corresponding results for weighted translation operators are stated below, but their proofs are left to the reader.

\begin{corollary}
A weighted translation operator $T_w$ on $C_0[1,\infty)$ is distributionally chaotic if and only if
\[
\inf_{k \in \N} \Big(\sup_{N \in \N} \frac{\card\{1 \leq n \leq N : |w(x) \cdots w(x+n-1)| > k \text{ for some } x \in [1,\infty)\}}{N}\Big) > 0.
\]
\end{corollary}

\begin{corollary}
Consider a weighted translation operator $T_w$ on $C_0(\R)$ with a zero-free weight function satisfying 
\[
\lim_{n \to \infty} \sup_{x \in (0,1)} |w(x-n) \cdots w(x-1)| = 0.
\]
Then, $T_w$ is distributionally chaotic if and only if
\[
\inf_{k \in \N} \Big(\sup_{N \in \N} \frac{\card\{1 \leq n \leq N : |w(x) \cdots w(x+n-1)| > k \text{ for some } x \in \R\}}{N}\Big) > 0.
\]
\end{corollary}

\begin{theorem}\label{DCSumC0X}
Consider a weighted composition operator $C_{w,f}$ on $C_0(\Omega)$ with positive weight function $w : \Omega \to (0,\infty)$.
Suppose that there is a sequence $(A_k)_{k \in \N}$ of relatively compact open sets in $\Omega$ such that 
the following properties hold:
\begin{itemize}
\item [\rm (a)] There exists a set $D \subset \N$ with $\udens(D) = 1$ such that 
  \[
  \lim_{n \in D} \|w^{(n)}\|_{f^{-n}(A_k)} = 0 \ \text{ for all } k \in \N.
  \]

\item [\rm (b)] There exists a set $E \subset \N$ with $\udens(E) > 0$ such that 
  \[
  \sum_{n \in E} \frac{1}{\|w^{(n)}\|_{f^{-n}(A_n)}} < \infty.
  \]
\end{itemize}
Then, $C_{w,f}$ is distributionally chaotic.
\end{theorem}

\begin{proof}
For each $k \in E$, choose a point $x_k \in f^{-k}(A_k)$ with $w^{(k)}(x_k) > \frac{1}{2} \|w^{(k)}\|_{f^{-k}(A_k)}$,
an open set $B_k$ in $\Omega$ with $f^k(x_k) \in B_k \subset \ov{B_k} \subset A_k$,
and a continuous map $\phi_k : \Omega \to [0,1]$ with 
\[
\phi_k = 1 \text{ on } \ov{B_k} \ \ \text{ and } \ \ \supp \phi_k \subset A_k.
\]
By property (a),
\[
\lim_{n \in D} \|(C_{w,f})^n(\phi_k)\| \leq \lim_{n \in D} \|w^{(n)}\|_{f^{-n}(A_k)} = 0 \ \text{ for all } k \in E. 
\]
By property (b),
\[
\sum_{n \in E} \frac{1}{\|w^{(n)}\|_{f^{-n}(B_n)}} < 2 \sum_{n \in E} \frac{1}{\|w^{(n)}\|_{f^{-n}(A_n)}} < \infty.
\]
Let us now prove that there exists $\varphi \in \ov{\spa\{\phi_k : k \in E\}}$ such that
\[
\lim_{n \in E} \|(C_{w,f})^n(\varphi)\| = \infty.
\]
In view of \cite[Proposition~8]{BBMP}, this will allow us to apply the Distributional Chaos Criterion and conclude that
$C_{w,f}$ is distributionally chaotic. Let 
\begin{align*}
 a_n &=\left\{
  \begin{array}{cl}
    \frac{1}{\|w^{(n)}\|_{f^{-n}(B_n)}}  &  \text{ if }  n \in E\\
    0  &  \text{ if } n \notin E,
  \end{array}
	     \right.\\
r_n &= \displaystyle \sum _{i\ge n} a_i\,,\\
c_n & =\left\{
  \begin{array}{cl}
    \frac{1}{\sqrt{r_n}\, \|w^{(n)}\|_{f^{-n}(B_n)}}  &  \text{ if }  n \in E\\
    0  &  \text{ if } n \notin E.
  \end{array}
	     \right.\\
\end{align*}
Calculations similar to those made in the proof of Theorem~\ref{DCSum} show that
\[
\sum_{n \in E} c_n < \infty \ \ \text{ and } \ \ \lim_{n \in E} c_n\, \|w^{(n)}\|_{f^{-n}(B_n)} = \infty. 
\]
Hence, we can define
\[
\varphi = \sum_{n \in E} c_n \phi_n \in C_0(\Omega),
\]
since this series is absolutely convergent. Moreover, $\varphi \in \ov{\spa\{\phi_k : k \in E\}}$. Finally,
 \begin{align*}
\lim_{n \in E} \|(C_{w,f})^n(\varphi)\| &= \lim_{n \in E} \Big\|\sum_{k \in E} c_k\, (C_{w,f})^n(\phi_k)\Big\|\\
  &\geq \lim_{n \in E} \|c_n\, (C_{w,f})^n(\phi_n)\|\\
  &\geq \lim_{n \in E} c_n\, \|w^{(n)}\|_{f^{-n}(B_n)} = \infty,
\end{align*}
which completes the proof.
\end{proof}

\begin{corollary} 
Consider a weighted composition operator $C_{w,f}$ on $C_0(\Omega)$ with positive weight function $w : \Omega \to (0,\infty)$.
If there is a relatively compact open set $A$ in $\Omega$ such that 
\begin{itemize}
\item there exists $D \subset \N$ with $\udens(D) = 1$ and $\displaystyle \lim_{n \in D} \|w^{(n)}\|_{f^{-n}(A)} = 0$, 

\item there exists $E \subset \N$ with $\udens(E) > 0$ and $\displaystyle \sum_{n \in E} \frac{1}{\|w^{(n)}\|_{f^{-n}(A)}} < \infty$,
\end{itemize}
then $C_{w,f}$ is distributionally chaotic.
\end{corollary}

\begin{theorem}\label{DDC-C0X}
Consider a weighted composition operator $C_{w,f}$ on $C_0(\Omega)$.
If the space $C_0(\Omega)$ is separable and 
\[
\lim_{n \to \infty} \|w^{(n)}\|_{f^{-n}(B)} = 0
\]
for every relatively compact open set $B$ in $\Omega$, then the following assertions are equivalent:
\begin{itemize}
\item [\rm (i)] $C_{w,f}$ is distributionally chaotic;
\item [\rm (ii)] $C_{w,f}$ is densely distributionally chaotic;
\item [\rm (iii)] $C_{w,f}$ admits a dense distributionally irregular manifold;
\item [\rm (iv)] There exist $\phi \in C_0(\Omega)$ and $\delta > 0$ such that 
  \[
  \udens\{n \in \N : \|(C_{w,f})^n(\phi)\| \geq \delta\} > 0.
  \]
\end{itemize}
\end{theorem}

\begin{proof}
(i) $\Rightarrow$ (iv): Take a distributionally irregular vector $\phi \in C_0(\Omega)$ for $C_{w,f}$.

\smallskip\noindent
(iv) $\Rightarrow$ (iii): The assumption implies that
\[
\lim_{n \to \infty} (C_{w,f})^n(\varphi) = 0 \ \text{ for all } \varphi \in C_c(\Omega).
\]
Since $C_c(\Omega)$ is dense in $C_0(\Omega)$, (iii) follows from (iv) and \cite[Theorem~33]{BBPW}.   

\smallskip\noindent
(iii) $\Rightarrow$ (ii) $\Rightarrow$ (i): Obvious.
\end{proof}

\begin{corollary} 
Assume the hypotheses of the previous theorem.
If there exist a constant $C > 0$ and a relatively compact open set $A$ in $\Omega$ such that
\[
\udens\{n \in \N : \|w^{(n)}\|_{f^{-n}(A)} \geq C \} > 0,
\]
then $C_{w,f}$ is densely distributionally chaotic.
\end{corollary}

\begin{proof}
Take any $\phi \in C_0(\Omega)$ with $\phi = 1$ on $A$, and put $\delta = C$.
Then, property (iv) in the previous theorem holds.
\end{proof}

\begin{theorem} 
Consider a weighted composition operator $C_{w,f}$ on $C_0(\Omega)$ with positive weight function $w : \Omega \to (0,\infty)$.
Suppose that the space $C_0(\Omega)$ is separable and the following properties hold:
\begin{itemize}
\item [\rm (a)] For every relatively compact open set $B$ in $\Omega$,
\[
\lim_{n \to \infty} \|w^{(n)}\|_{f^{-n}(B)} = 0.
\]
\item [\rm (b)] There exist a relatively compact open set $A$ in $\Omega$ and a set $D \subset \N$ with $\udens(D) > 0$ such that
\[
\sum_{n \in D} \frac{1}{\|w^{(n)}\|_A} < \infty.
\]
\end{itemize}
Then, $C_{w,f}$ is densely distributionally chaotic.
\end{theorem}

\begin{proof}
For each $n \in \N$, let $A_n$ be a relatively compact open set in $\Omega$ such that $f^n(A) \subset A_n$. Since
\[
\sum_{n \in D} \frac{1}{\|w^{(n)}\|_{f^{-n}(A_n)}} \leq \sum_{n \in D} \frac{1}{\|w^{(n)}\|_A} < \infty,
\]
it follows from Theorem~\ref{DCSumC0X} that $C_{w,f}$ is distributionally chaotic.
Hence, by Theorem~\ref{DDC-C0X}, $C_{w,f}$ is densely distributionally chaotic.
\end{proof}

The following results on weighted shifts follow easily from the previous theorem (we omit the details), 
but they can also be derived from Corollaries~\ref{dcuws-c0} and~\ref{dcbws-c0}.

\begin{corollary} 
Consider a weighted shift $B_w$ on $c_0(\N)$ with positive weights. 
If there exists a set $D \subset \N$ with $\udens(D) > 0$ such that
\[
\sum_{n \in D} \left(\prod_{j=1}^{n} w_j \right)^{-1} < \infty,
\]
then $B_w$ is densely distributionally chaotic.
\end{corollary}

\begin{corollary}
Consider a weighted shift $B_w$ on $c_0(\Z)$ with positive weights. If 
\[
\lim_{n \to \infty} \prod_{j=-n+1}^{0} w_j = 0
\]
and there exists a set $D \subset \N$ with $\udens(D) > 0$ such that
\[
\sum_{n \in D} \left(\prod_{j=1}^{n} w_j \right)^{-1}  < \infty,
\]
then $B_w$ is densely distributionally chaotic.
\end{corollary}

%%%%%%%%%%%%%%%%%%%%%%%%%%%%%%%%%%%%%%%%%%%%%%%%%%%%%%%%%%%%%%%%%%%%%%
%%%%%%%%%%%%%%%%%%%%%%%%%%%%%%%%%%%%%%%%%%%%%%%%%%%%%%%%%%%%%%%%%%%%%%

\section{Absolutely Ces\`aro  bounded weighted composition operators}

Recall that an operator $T$ on a Banach space $Y$ is said to be {\em $p$-absolutely Ces\`aro bounded} if 
there exists $C \in (0,\infty)$ such that 
\[
\sup_{N \in \N} \frac{1}{N} \sum_{n=1}^N \|T^ny\|^p \leq C \|y\|^p \ \text{ for all } y \in Y.
\]
More generally, given a subspace $Z$ of $Y$, we say that $T$ is {\em $p$-absolutely Ces\`aro bounded in $Z$}
if the above inequality holds for every $y \in Z$. By defining the extended real number 
\[
N_p(T) = \sup_{\|y\| = 1} \sup_{N \in \N} \frac{1}{N} \sum_{n=1}^N \|T^n y\|^p,
\]
we have that 
\[
T \text{ is $p$-absolutely Ces\`aro bounded } \ \Longleftrightarrow \ N_p(T) < \infty.
\]
It is usual to say {\em absolutely Ces\`aro bounded} instead of $1$-absolutely Ces\`aro bounded.

Recall also that an operator $T$ on a Banach space $Y$ is said to be {\em mean ergodic} if the sequence $(M_n(T))_{n \in \N}$
of \emph{Cesàro means} of $T$, defined by
\[
M_n(T)y =\frac{1}{n+1}\sum_{k=0}^n T^ky \ \ \ \ \ (n \in \N),
\]
converges in the strong operator topology of the space of all operators on $Y$.

%%%%%%%%%%%%%%%%%%%%%%%%%%%%%%
%%%%%%%%%%%%%%%%%%%%%%%%%%%%%%

\subsection{The case of the space $L^p(\mu)$}

\begin{theorem}\label{p-ACB}
For any weighted composition operator $C_{w,f}$ on $L^p(\mu)$,
\begin{equation}\label{p-ACB-1}
N_p(C_{w,f}) = \sup_{0 < \mu(B) < \infty} \sup_{N \in \N} \frac{1}{N} \sum_{n=1}^N \frac{\mu_n(f^{-n}(B))}{\mu(B)}\cdot
\end{equation}
In particular, $C_{w,f}$ is $p$-absolutely Ces\`aro bounded if and only if there exists a constant $C \in (0,\infty)$ such that,
for each measurable set $B$ of finite positive $\mu$-measure,
\begin{equation}\label{p-ACB-2}
 \frac{1}{N} \sum _{n=1}^N \mu_n(f^{-n}(B)) \leq C\, \mu (B) \ \text{ for all } N \in \N.
\end{equation}
\end{theorem}

\begin{proof}
Denote the right-hand side of (\ref{p-ACB-1}) by $r$.

Given a measurable set $B$ of finite positive $\mu$-measure, define $\phi = \frac{1}{\mu(B)^{1/p}}\,\rchi_{B}$. 
Since $\|\phi\|_p =1$, we obtain
\[
\sup_{N \in \N} \frac{1}{N} \sum_{n=1}^N \frac{\mu_n(f^{-n}(B))}{\mu(B)}
  = \sup_{N \in \N} \frac{1}{N} \sum_{n=1}^N \|(C_{w,f})^n(\phi)\|_p^p \leq  N_p(C_{w,f}).
\]
This shows that $r \leq N_p(C_{w,f})$.

Conversely, fix $t > 1$. Given $\varphi \in L^p(\mu)$ with $\|\varphi\|_p = 1$, consider the measurable sets
\[
B_i = \{x \in X : t^{i-1} \leq |\varphi(x)| < t^i\} \ \ \  (i \in \Z).
\]
Then, for every $N \in \N$,
\begin{eqnarray*}
\frac{1}{N} \sum_{n=1}^N \|(C_{w,f})^n(\varphi)\|_p^p 
& = & \frac{1}{N} \sum_{n=1}^N \sum_{i \in \Z} \int_{f^{-n}(B_i)} |\varphi \circ f^n|^p |w \circ f^{n-1}|^p \cdots |w \circ f|^p |w|^p d\mu\\
& \leq & \sum_{i \in \Z} t^{ip} \frac{1}{N} \sum_{n=1}^N \mu_n( f^{-n}(B_i)) \leq \sum _{i \in \Z} t^{ip} r \mu(B_i)\\
& = & t^p r \sum _{i \in \Z} t^{(i-1)p} \mu(B_i) \leq t^p r \|\varphi\|_p^p = t^p r.
\end{eqnarray*}
This shows that $N_p(C_{w,f}) \leq t^p r$. Since $t > 1$ is arbitrary, we obtain $N_p(C_{w,f}) \leq r$.

The last assertion follows from (\ref{p-ACB-1}).
\end{proof}

\begin{corollary}
Consider a weighted composition operator $C_{w,f}$ on $L^p(\mu)$.
If $p > 1$ and there exists $C \in (0,\infty)$ such that, for each measurable set $B$ of finite positive $\mu$-measure,
\[
\frac{1}{N} \sum_{n=1}^N \mu_n(f^{-n}(B)) \leq C\, \mu (B) \ \text{for all } N \in \N,
\]
then $C_{w,f}$ is mean ergodic.
\end{corollary}

\begin{proof}
Every $p$-absolutely Ces\`aro bounded operator in a reflexive Banach space is mean ergodic \cite[Corollary~2.7]{BermBMP}.
\end{proof}

\begin{corollary}\label{p-ACB-UWBS}
For any weighted shift $B_w$ on $\ell^p(\N)$, 
\begin{equation}\label{p-ACB-3}
N_p(B_w) = \sup_{i \in \N, N \in \N} \frac{1}{N} \sum_{n=1}^{\min\{N,\;i-1\}} |w_{i-n} \cdots w_{i-1}|^p.
\end{equation}
In particular, $B_w$ is $p$-absolutely Ces\`aro bounded if and only if the above supremum is finite.
\end{corollary}

\begin{proof}
Formula (\ref{p-ACB-3}) is a special case of formula (\ref{p-ACB-1}). Indeed, regard $B_w$ as $C_{w,f}$ by considering 
$X = \N$, $\fM = P(\N)$, $\mu$ the counting measure on $\fM$ and $f : n \in \N \mapsto n+1 \in \N$.
Denote the right-hand side of (\ref{p-ACB-3}) by $r$. By (\ref{p-ACB-1}),
\[
N_p(B_w) \geq  \sup_{i \in \N} \sup_{N \in \N} \frac{1}{N} \sum_{n=1}^N \frac{\mu_n(f^{-n}(\{i\}))}{\mu(\{i\})} = r.
\]
On the other hand, for any nonempty finite set $B = \{i_1,\ldots,i_k\} \subset \N$,
\begin{align*}
\sup_{N \in \N} \frac{1}{N} \sum_{n=1}^N \frac{\mu_n(f^{-n}(B))}{\mu(B)}
& = \sup_{N \in \N} \frac{1}{N} \sum_{n=1}^N \sum_{j=1}^k \frac{\mu_n(f^{-n}(\{i_j\}))}{k}\\
& \leq \sum_{j=1}^k \frac{1}{k} \Big(\sup_{N \in \N} \frac{1}{N} \sum_{n=1}^{\min\{N,\, i_j - 1\}} |w_{i_j-n} \cdots w_{i_j-1}|^p\Big) \leq r.
\end{align*}
Thus, by (\ref{p-ACB-1}), $N_p(B_w) \leq r$.
\end{proof}

\begin{corollary}\label{p-ACB-BWBS}
For any weighted shift $B_w$ on $\ell^p(\Z)$, 
\begin{equation}\label{p-ACB-5}
N_p(B_w) = \sup_{i \in \Z, N \in \N} \frac{1}{N} \sum_{n=1}^N |w_{i-n} \cdots w_{i-1}|^p.
\end{equation}
In particular, $B_w$ is $p$-absolutely Ces\`aro bounded if and only if the above supremum is finite.
\end{corollary}

\begin{proof}
Formula (\ref{p-ACB-5}) is also a special case of formula (\ref{p-ACB-1}).
The proof is similar to that of Corollary~\ref{p-ACB-UWBS} and is left to the reader.
\end{proof}

A special case of Jensen's inequality (see \cite[Theorem~3.3]{WRud87}) asserts that
\[
\varphi\Big(\frac{a_1 + \cdots + a_N}{N}\Big) \leq \frac{\varphi(a_1) + \cdots + \varphi(a_N)}{N}\,,
\]
whenever $\varphi : [0,\infty) \to \R$ is a convex function, $N \in \N$ and $a_1,\ldots,a_N \in [0,\infty)$.
As a consequence, for any operator $T$ on a Banach space $Y$,
\[
T \text{ $p$-absolutely Ces\`aro bounded } \ \Longrightarrow \ T \text{ $q$-absolutely Ces\`aro bounded } \forall q \in [1,p]
\]
(for a generalization of this fact see \cite[Theorem~7.12]{ABeYa}).

\begin{example}\label{ejemplos}
Given any real number $q > p$, consider the weighted shift $B_w$ on $\ell^p(\N)$ whose weight sequence
$w = (w_n)_{n \in \N}$ is given by
\[
w_n = \left(\frac{n+1}{n}\right)^{\frac{1}{q}} \ \text{ for all } n \in \N.
\]
Then, $B_w$ is $p$-absolutely Ces\`aro bounded, but it is not $q$-absolutely Ces\`aro bounded.
In particular, $B_w$ is not power-bounded.
\end{example}

\begin{proof}
Indeed, let $\alpha = \frac{p}{q} \in (0,1)$. For any $i \geq 2$ and $N \in \N$,
\begin{eqnarray*}
\sum_{n=1}^{\min\{N,\;i-1\}} |w_{i-n} \cdots w_{i-1}|^p
&=& \sum_{n=1}^{\min\{N,\;i-1\}} \Big(\frac{i}{i-n}\Big)^\alpha\\
&=& \left\{
  \begin{array}{cl}
  i^\alpha \sum_{n=1}^{\min\{N,\;i-1\}} (i-n)^{-\alpha} & \text{if } i \le 2N\\
  \sum_{n=1}^N \Big(\frac{i}{i-n}\Big)^\alpha  & \text{if } i > 2N
  \end{array}
  \right.\\
&\le& \left\{
  \begin{array}{cl} 
  i^\alpha \sum_{n=1}^{i-1} n^{-\alpha} & \text{if } i \le 2N\\
  \sum_{n=1}^N \Big(\frac{i}{i-N}\Big)^\alpha & \text{if } i > 2N
  \end{array}
  \right.\\
&\le& \left\{
  \begin{array}{cl} 
  i^\alpha \big(1 + \int_1^{i-1} t^{-\alpha} dt \big) & \text{if } i \le 2N\\
  \sum_{n=1}^N 2^\alpha & \text{if } i > 2N
  \end{array}
  \right.\\
&\le& \left\{
  \begin{array}{ll}
  \frac{i}{1 - \alpha} & \text{if } i \le 2N\\
  2N & \text{if } i > 2N
  \end{array}
  \right.\\
&\leq& \frac{2}{1 - \alpha}\, N.
\end{eqnarray*}
Thus, by Corollary~\ref{p-ACB-UWBS}, $B_w$ is $p$-absolutely Ces\`aro bounded.

On the other hand, suppose that $B_w$ is $q$-absolutely Ces\`aro bounded.
Then, there exists a constant  $C \in (0,\infty)$ such that 
\[
\sup_{N \in \N} \frac{1}{N} \sum_{n=1}^N \|(B_w)^n x\|_p^q \leq C \|x\|_p^q \ \text{ for all } x \in \ell_p(\N).
\]
In particular, for each $i \geq 2$,
\[
C \geq \frac{1}{i-1} \sum_{n=1}^{i-1} \|(B_w)^n e_i\|_p^q = \frac{1}{i-1} \Big(\frac{i}{i-1} + \frac{i}{i-2} + \cdots + \frac{i}{1}\Big)
= \frac{i}{i-1} \sum_{n=1}^{i-1} \frac{1}{n}\cdot
\]
Since the harmonic series diverges, we have a contradiction.
\end{proof}

The above example comes from \cite[Theorem~2.1]{BermBMP}, 
where it was proved that $B_w$ is $p$-absolutely Ces\`aro bounded but not power-bounded.
Here we obtain the stronger conclusion that $B_w$ is not $q$-absolutely Ces\`aro bounded.
Moreover, the proof that $B_w$ is $p$-absolutely Ces\`aro bounded, as an application of Corollary~\ref{p-ACB-UWBS},
is slightly shorter (but it uses similar estimates).

\begin{example} 
The weighted translation operator $T_w$ on $L^p[1,\infty)$ with weight function
\[
w(x) = \Big(\frac{x+1}{x}\Big)^{\frac{1-\varepsilon}{p}},
\]
where $\eps > 0$ is fixed, is $p$-absolutely Ces\`aro bounded.
\end{example}

\begin{proof}
In view of Theorem~\ref{p-ACB}, it is enough to show that there exists a constant $C \in (0,\infty)$ such that,
for each Lebesgue measurable set $B \subset [1, \infty)$ with finite positive Lebesgue measure,
\[
\frac{1}{N} \sum _{n=1}^N \mu_n(f^{-n}(B)) \leq C\, \mu (B) \ \text{ for all } N \in \N.
\]
But, in fact,
\begin{align*}
\sum _{n=1}^N \mu_n(f^{-n}(B)) 
&= \sum_{n=1}^N \int_{(B-n) \cap [1,\infty)} \Big(\frac{x+n}{x}\Big)^{1-\eps} dx \\
&= \sum _{n=1}^N \int_{B \cap [n+1,\infty) } \Big(\frac{x}{x-n}\Big)^{1-\eps} dx =
  \int_{B} \sum_{n=1}^{\min\{N,\;[x]-1\}} \Big(\frac{x}{x-n}\Big)^{1-\eps} dx \\
&\leq \int_{B \cap [1,1+2N]} x^{1-\eps} \sum_{n=1}^{[x]-1} \Big(\frac{1}{x-n}\Big)^{1-\eps} dx
  + \int_{B \cap [1+2N, \infty)} \sum_{n=1}^N \Big(\frac{x}{x-n}\Big)^{1-\eps} dx \\
&< \frac{1+2N}{\eps} \mu(B) + 2N \mu(B) \leq (2+ \frac{3}{\eps}) N \mu(B),
\end{align*}
as it was to be shown.    
\end{proof}

%%%%%%%%%%%%%%%%%%%%%%%%%%%%%%
%%%%%%%%%%%%%%%%%%%%%%%%%%%%%%

\subsection{The case of the space $C_0(\Omega)$}

\begin{theorem}\label{p-ACB-C0X}
For any weighted composition operator $C_{w,f}$ on $C_0(\Omega)$,
\begin{equation}\label{p-ACB-C0X-1}
N_p(C_{w,f}) = \sup_{N \in \N} \frac{1}{N} \sum_{n=1}^N \|w^{(n)}\|^p.
\end{equation}
In particular, $C_{w,f}$ is $p$-absolutely Ces\`aro bounded if and only if there exists a constant $C \in (0,\infty)$ such that
\begin{equation}\label{p-ACB-C0X-2}
\frac{1}{N} \sum _{n=1}^N \|w^{(n)}\|^p \leq C \ \text{ for all } N \in \N.
\end{equation}
\end{theorem}

\begin{proof}
Denote the right-hand side of (\ref{p-ACB-C0X-1}) by $r$.

For every $\varphi \in C_0(\Omega)$ with $\|\varphi\| = 1$,
\[
\frac{1}{N} \sum_{n=1}^N \|(C_{w,f})^n(\varphi)\|^p 
  = \frac{1}{N} \sum_{n=1}^N \|(\varphi \circ f^n) \cdot w^{(n)}\|^p \leq  \frac{1}{N} \sum_{n=1}^N \|w^{(n)}\|^p.
\]
This shows that $N_p(C_{w,f}) \leq r$.

Conversely, fix $N \in \N$ and $\eps > 0$. 
For each $n \in \{1,\ldots,N\}$, there is $x_n \in \Omega$ with $|w^{(n)}(x_n)| \geq \|w^{(n)}\| - \eps$.
Let $V$ be a relatively compact open set in $\Omega$ such that $\{f^n(x_1),\ldots,f^n(x_N)\} \subset V$.
By Urysohn's lemma, there exists a continuous map $\phi : \Omega \to [0,1]$ such that $\supp \phi \subset V$
and $\phi(f^n(x_n)) = 1$ for all $n \in \{1,\ldots,N\}$. Hence,
\[
N_p(C_{w,f}) \geq \frac{1}{N} \sum_{n=1}^N \|(C_{w,f})^n(\phi)\|^p 
    \geq \frac{1}{N} \sum_{n=1}^N |\phi(f^n(x_n)) w^{(n)}(x_n)|^p 
    \geq \frac{1}{N} \sum_{n=1}^N (\|w^{(n)}\| - \eps)^p.
\]
Since $N \in \N$ and $\eps > 0$ are arbitrary, we obtain $N_p(C_{w,f}) \geq r$.

The last assertion follows from (\ref{p-ACB-C0X-1}).
\end{proof}

The results below follow immediately from the previous theorem.

\begin{corollary}\label{p-ACB-UWBS-C0X}
Let $\Omega = \N$ or $\Z$. For any weighted shift $B_w$ on $c_0(\Omega)$,  
\begin{equation}\label{p-ACB-C0X-3}
N_p(B_w) = \sup_{N \in \N} \frac{1}{N} \sum_{n=1}^N \Big(\sup_{i \in \Omega} |w_i \cdots w_{i+n-1}|\Big)^p.
\end{equation}
In particular, $B_w$ is $p$-absolutely Ces\`aro bounded if and only if the above supremum is finite.
\end{corollary}

\begin{corollary}\label{p-ACB-BWBS-C0X}
Let $\Omega = [1,\infty)$ or $\R$. For any weighted translation operator $T_w$ on $C_0(\Omega)$,  
\begin{equation}\label{p-ACB-C0X-5}
N_p(T_w) = \sup_{N \in \N} \frac{1}{N} \sum_{n=1}^N \Big(\sup_{x \in \Omega} |w(x) \cdots w(x+n-1)|\Big)^p.
\end{equation}
In particular, $T_w$ is $p$-absolutely Ces\`aro bounded if and only if the above supremum is finite.
\end{corollary}

\begin{remark}
Formula (\ref{p-ACB-3}) and formula (\ref{p-ACB-C0X-3}) for $\Omega = \N$ do not give the same value in general.
For instance, consider the weight sequence $w = (w_n)_{n \in \N}$ obtained by concatenating the blocks
\[
z^{(n)} = \big(\underbrace{\sqrt[n]{e},\ldots,\sqrt[n]{e}}_{n \text{ times}},\frac{1}{e}\big) \ \ \text{ for } n \geq 2.
\]
Since
\[
\sup_{i \in \N} |w_i \cdots w_{i+n-1}| = \left\{\begin{array}{cl} \sqrt{e} & \text{if } n = 1\\ e & \text{if } n \geq 2, \end{array} \right.
\]
we have that
\[
\sup_{N \in \N} \frac{1}{N} \sum_{n=1}^N \Big(\sup_{i \in \N} |w_i \cdots w_{i+n-1}|\Big)^p
  = \sup_{N \in \N} \frac{(\sqrt{e})^p + (N-1) e^p}{N} = e^p.
\]
On the other hand, the largest possible value $V_N$ for
\[
\frac{1}{N} \sum_{n=1}^{\min\{N,\;i-1\}} |w_{i-n} \cdots w_{i-1}|^p
\]
is given by
\[
V_N = \left\{\begin{array}{cl} (\sqrt{e})^p & \text{if } N = 1\\ 
  \displaystyle \frac{(\sqrt[N]{e}\,)^p + (\sqrt[N]{e}\,)^{2p} + \cdots + (\sqrt[N]{e}\,)^{Np}}{N} & \text{if } N \geq 2.
  \end{array} \right.
\]
In particular,
\[
V_N < e^p \ \ \text{ for all } N \geq 1.
\]
Since
\begin{align*}
\lim_{N \to \infty} V_N 
  &= \lim_{N \to \infty} \frac{(\sqrt[N]{e}\,)^p (e^p - 1)}{N ((\sqrt[N]{e}\,)^p - 1)} 
     = \lim_{t \to 0^+} \frac{t\, e^{pt} (e^p - 1)}{e^{pt} - 1}\\
  &= (e^p - 1) \lim_{t \to 0^+} \frac{t}{e^{pt} - 1} = (e^p - 1) \lim_{t \to 0^+} \frac{1}{p\, e^{pt}} = \frac{e^p - 1}{p} < e^p,
\end{align*}
we conclude that
\[
\sup_{i \in \N, N \in \N} \frac{1}{N} \sum_{n=1}^{\min\{N,\;i-1\}} |w_{i-n} \cdots w_{i-1}|^p = \sup_{N \in \N} V_N < e^p.
\]

Similarly, formula (\ref{p-ACB-5}) and formula (\ref{p-ACB-C0X-3}) for $\Omega = \Z$ do not give the same value in general.
\end{remark}

%%%%%%%%%%%%%%%%%%%%%%%%%%%%%%%%%%%%%%%%%%%%%%%%%%%%%%%%%%%%%%%%%%%%%%
%%%%%%%%%%%%%%%%%%%%%%%%%%%%%%%%%%%%%%%%%%%%%%%%%%%%%%%%%%%%%%%%%%%%%%

\section{Mean Li-Yorke chaotic weighted composition operators}

Given a metric space $M$, recall that a map $f : M \to M$ is said to be {\em mean Li-Yorke chaotic} if there exists an uncountable set 
$S \subset M$ such that each pair $(x,y)$ of distinct points in $S$ is a {\em mean Li-Yorke pair for $f$}, in the sense that
\[
\liminf_{N \to \infty} \frac{1}{N} \sum_{n=1}^N d(f^n(x),f^n(y)) = 0 
\ \ \text{ and } \ \ 
\limsup_{N \to \infty} \frac{1}{N} \sum_{n=1}^N d(f^n(x),f^n(y)) > 0.
\]
If the set $S$ can be chosen to be dense in $M$, then $f$ is {\em densely mean Li-Yorke chaotic}.

An extensive study of the concept of mean Li-Yorke chaos in the setting of linear dynamics was developed in \cite{BBP}.
In particular, the following useful characterizations were obtained:
For any operator $T$ on any Banach space $Y$, the following assertions are equivalent:
\begin{itemize}
\item [(i)] $T$ is mean Li-Yorke chaotic;
\item [(ii)] $T$ admits an {\em absolutely mean semi-irregular vector}, that is, a vector $y \in Y$ such that
  \[
  \liminf_{N \to \infty} \frac{1}{N} \sum_{n=1}^N \|T^ny\| = 0 
  \ \ \text{ and } \ \ 
  \limsup_{N \to \infty} \frac{1}{N} \sum_{n=1}^N \|T^ny\| > 0.
  \]
\item [(iii)] $T$ admits an {\em absolutely mean irregular vector}, that is, a vector $y \in Y$ such that
  \[
  \liminf_{N \to \infty} \frac{1}{N} \sum_{n=1}^N \|T^ny\| = 0 
  \ \ \text{ and } \ \ 
  \limsup_{N \to \infty} \frac{1}{N} \sum_{n=1}^N \|T^ny\| = \infty.
  \]
\end{itemize}

Let us also recall that by an {\em absolutely mean irregular manifold for $T$} we mean a vector subspace of $Y$ consisting,
except for the zero vector, of absolutely mean irregular vectors for $T$.

%%%%%%%%%%%%%%%%%%%%%%%%%%%%%%
%%%%%%%%%%%%%%%%%%%%%%%%%%%%%%

\subsection{The case of the space $L^p(\mu)$}

\begin{theorem}[Necessary Condition]\label{MLY-Nec}
If a weighted composition operator $C_{w,f}$ on $L^p(\mu)$ is mean Li-Yorke chaotic, 
then there exists a nonempty countable family $(B_i)_{i \in I}$ of measurables sets of finite positive $\mu$-measure such that
\begin{equation}\label{MLY-eq1}
\liminf_{N \to \infty} \frac{1}{N} \sum_{n=1}^N \mu_n(f^{-n}(B_i))^\frac{1}{p} = 0 \ \ \text{ for all } i \in I
\end{equation}
and
\begin{equation}\label{MLY-eq2}
\sup \Big\{\frac{1}{N} \sum_{n=1}^N \frac{\mu_n(f^{-n}(B_i))}{\mu(B_i)} : i \in I, N \in \N \Big\} = \infty.
\end{equation}
\end{theorem}

\begin{proof}
Let $\varphi \in L^p(\mu)$ be an absolutely mean irregular vector for $C_{w,f}$. 
Consider the measurable sets
\[
B_i = \{x \in X : 2^{i-1} \leq |\varphi(x)| < 2^i\} \ \ \ (i \in \Z)
\]
and let $I$ be the nonempty subset of $\Z$ given by $I = \{i \in \Z : \mu(B_i) > 0\}$.
We have that $0 < \mu(B_i) < \infty$ for all $i \in I$, because
\[
2^{(i-1)p} \mu(B_i) \leq \int_X |\varphi|^p d\mu <\infty \ \ \ (i \in \Z).
\]
Since
\begin{align*}
2^{i-1} \frac{1}{N} \sum_{n=1}^N \mu_n(f^{-n}(B_i))^\frac{1}{p} 
  &\leq \frac{1}{N} \sum_{n=1}^N \Big(\int_{f^{-n}(B_i)} |\varphi \circ f^n|^p |w^{(n)}|^p d\mu\Big)^\frac{1}{p}\\
  &\leq \frac{1}{N} \sum_{n=1}^N \|(C_{w,f})^n(\varphi)\|_p
\end{align*}
and $\varphi$ is an absolutely mean irregular vector for $C_{w,f}$, it follows that (\ref{MLY-eq1}) holds.
On the other hand, if (\ref{MLY-eq2}) fails, that is,
\[
C = \sup \Big\{\frac{1}{N} \sum_{n=1}^N \frac{\mu_n(f^{-n}(B_i))}{\mu(B_i)} : i \in I, N \in \N \Big\} < \infty,
\]
then
\begin{align*}
\Big(\frac{1}{N} \sum_{n=1}^N \|(C_{w,f})^n(\varphi)\|_p\Big)^p 
&\leq \frac{1}{N} \sum_{n=1}^N \|(C_{w,f})^n(\varphi)\|_p^p\\
&= \frac{1}{N} \sum_{n=1}^N \sum_{i \in \Z} \int_{f^{-n}(B_i)} |\varphi \circ f^n|^p |w^{(n)}|^p d\mu\\
&\leq \frac{1}{N} \sum_{n=1}^N \sum_{i \in \Z} 2^{ip} \mu_n(f^{-n}(B_i))\\
&= \sum_{i \in \Z} 2^{ip} \frac{1}{N} \sum_{n=1}^N \mu_n(f^{-n}(B_i))\\
&\leq \sum_{i \in \Z} 2^{ip} C \mu(B_i) = 2^p C \sum_{i \in \Z} 2^{(i-1)p} \mu(B_i) \leq 2^p C \|\varphi\|_p^p,
\end{align*}
contradicting the fact that $\varphi$ is an absolutely mean irregular vector for $C_{w,f}$.
\end{proof}

\begin{remark}
Note that the countable family $(B_i)_{i \in I}$ constructed in the proof of Theorem~\ref{MLY-Nec} has the following additional properties:
the sets $B_i$ are pairwise disjoint and there exists an increasing sequence $(N_j)_{j \in \N}$ of positive integers such that
\[
\lim_{j \to \infty} \frac{1}{N_j} \sum_{n=1}^{N_j} \mu_n(f^{-n}(B_i))^\frac{1}{p} = 0 \ \ \text{ for all } i \in I.
\]
\end{remark}

\begin{theorem}[Sufficient Condition]\label{MLY-Suf}
Consider a weighted composition operator $C_{w,f}$ on $L^p(\mu)$.
If there exists a nonempty countable family $(B_i)_{i \in I}$ of measurables sets of finite positive $\mu$-measure such that
\begin{equation}\label{MLY-eq1b}
\liminf_{N \to \infty} \frac{1}{N} \sum_{n=1}^N \mu_n(f^{-n}(B_i))^\frac{1}{p} = 0 \ \ \text{ for all } i \in I
\end{equation}
and
\begin{equation}\label{MLY-eq2b}
\sup \Big\{\frac{1}{N} \sum_{n=1}^N \frac{\mu_n(f^{-n}(B_i))^\frac{1}{p}}{\mu(B_i)} : i \in I, N \in \N \Big\} = \infty,
\end{equation}
then $C_{w,f}$ is mean Li-Yorke chaotic
\end{theorem}

\begin{proof}
Let $\fX$ be the closed subspace of $L^p(\mu)$ generated by the characteristic functions $\rchi_B$ such that 
$B$ is a measurable set of finite $\mu$-measure satisfying
\begin{equation}\label{MLY-eq0}
\liminf_{N \to \infty} \frac{1}{N} \sum_{n=1}^N \mu_n(f^{-n}(B))^\frac{1}{p} = 0.
\end{equation}
For each $i \in I$, let $\phi_i = \frac{1}{\mu(B_i)}\, \rchi_{B_i}$. By (\ref{MLY-eq1b}), $\phi_i \in \fX$ for all $i \in I$. 
Since $\|\phi_i\|_p = 1$ and 
\[
\frac{1}{N} \sum_{n=1}^N \|(C_{w,f})^n(\phi_i)\|_p = \frac{1}{N} \sum_{n=1}^N \frac{\mu_n(f^{-n}(B_i))^\frac{1}{p}}{\mu(B_i)},
\]
it follows from (\ref{MLY-eq2b}) that $C_{w,f}$ is not absolutely Ces\`aro bounded in $\fX$.
Suppose, by contradiction, that $C_{w,f}$ is not mean Li-Yorke chaotic. 
Then, $C_{w,f}$ does not admit an absolutely mean semi-irregular vector, and so (\ref{MLY-eq0}) is equivalent to
\[
\lim_{N \to \infty} \frac{1}{N} \sum_{n=1}^N \mu_n(f^{-n}(B))^\frac{1}{p} = 0.
\]
This implies that the set $\cR_1$ of all $\varphi \in \fX$ such that
\[
\liminf_{N \to \infty} \frac{1}{N} \sum_{n=1}^N \|(C_{w,f})^n(\varphi)\|_p = 0
\]
is residual in $\fX$. Since $C_{w,f}$ is not absolutely Ces\`aro bounded in $\fX$, 
\cite[Theorem~4]{BBP} implies that the set $\cR_2$ of all $\varphi \in \fX$ such that
\[
\limsup_{N \to \infty} \frac{1}{N} \sum_{n=1}^N \|(C_{w,f})^n(\varphi)\|_p = \infty
\]
is also residual in $\fX$. Since each $\varphi \in \cR_1 \cap \cR_2$ is an absolutely mean irregular vector for $C_{w,f}$,
we conclude that $C_{w,f}$ is mean Li-Yorke chaotic, a contradiction. 
\end{proof}

In the case $p = 1$, conditions (\ref{MLY-eq2}) and (\ref{MLY-eq2b}) coincide, and so Theorems~\ref{MLY-Nec} and~\ref{MLY-Suf} 
gives us a characterization of the mean Li-Yorke chaotic weighted composition operators on $L^1(\mu)$.
More precisely, the following result holds.

\begin{corollary}\label{MLY}
Assume $p = 1$. Given a weighted composition operator $C_{w,f}$ on $L^1(\mu)$, let $\fX$ be the closed subspace of $L^1(\mu)$ 
generated by the characteristic functions $\rchi_B$ such that $B$ is a measurable set of finite $\mu$-measure satisfying
\[
\liminf_{N \to \infty} \frac{1}{N} \sum_{n=1}^N \mu_n(f^{-n}(B)) = 0.
\]
Then, the following assertions are equivalent:
\begin{itemize}
\item [\rm (i)] $C_{w,f}$ is mean Li-Yorke chaotic;
\item [\rm (ii)] There exists a nonempty countable family $(B_i)_{i \in I}$ of measurables sets of finite positive $\mu$-measure such that
  \[
  \liminf_{N \to \infty} \frac{1}{N} \sum_{n=1}^N \mu_n(f^{-n}(B_i)) = 0 \ \ \text{ for all } i \in I
  \]
  and
  \[
  \sup \Big\{\frac{1}{N} \sum_{n=1}^N \frac{\mu_n(f^{-n}(B_i))}{\mu(B_i)} : i \in I, N \in \N \Big\} = \infty.
  \]
\item [\rm (iii)] $C_{w,f}$ is not absolutely Ces\`aro bounded in $\fX$.
\end{itemize}
\end{corollary}

\begin{theorem}\label{DMLY}
Consider a weighted composition operator $C_{w,f}$ on $L^p(\mu)$. 
If for every measurable set $A$ of finite $\mu$-measure and for every $\eps > 0$, there is a measurable set $B \subset A$ with 
\[
\mu(A \backslash B) < \eps \ \ \text{ and } \ \ \ \liminf_{N \to \infty} \frac{1}{N} \sum_{n=1}^N \mu_{n}(f^{-n}(B))^\frac{1}{p} = 0,
\]
then the following assertions are equivalent:
\begin{itemize}
\item [\rm (i)] $C_{w,f}$ is mean Li-Yorke chaotic;
\item [\rm (ii)] $C_{w,f}$ has a residual set of absolutely mean irregular vectors;
\item [\rm (iii)] $C_{w,f}$ is not absolutely Ces\`aro bounded.
\end{itemize}
\end{theorem}

\begin {proof} 
Let $\fX_0$ be the set of all simple functions of the form $\sum _{k=1}^m b_k \rchi_{B_k}$, 
where $b_1,\ldots,b_m$ are scalars and $B_1,\ldots,B_m$ are measurable sets of finite $\mu$-measure satisfying
\begin{equation}\label{MLY-eq3}
\liminf_{N \to \infty} \frac{1}{N} \sum_{n=1}^N \mu_n(f^{-n}(B_1 \cup \ldots \cup B_m))^\frac{1}{p} = 0.
\end{equation}
By the assumption, $\fX_0$ is dense in $L^p(\mu)$. Moreover, by (\ref{MLY-eq3}),
\[
\liminf_{N \to \infty} \frac{1}{N} \sum_{n=1}^N \|(C_{w,f})^n(\varphi)\|_p = 0 \ \text{ for all } \varphi \in \fX_0.
\]
Hence, the equivalences between properties (i), (ii) and (iii) follow from \cite[Theorem~22]{BBP}.
\end{proof}

\begin{remark}\label{R-DMLY-PB}
{\bf (a)} In the case $L^p(\mu)$ is separable, \cite[Theorem~17]{BBP} says that (ii) is equivalent~to
\begin{itemize}
\item [(ii')] $C_{w,f}$ is densely mean Li-Yorke chaotic.
\end{itemize}

\noindent
{\bf (b)} If the space $L^p(\mu)$ is separable and for every measurable set $A$ of finite $\mu$-measure and for every $\eps > 0$, 
there exists a measurable set $B \subset A$ with 
\[
\mu(A \backslash B) < \eps \ \ \text{ and } \ \ \ \lim_{N \to \infty} \frac{1}{N} \sum_{n=1}^N \mu_{n}(f^{-n}(B))^\frac{1}{p} = 0,
\]
then \cite[Theorem~29]{BBP} implies that (i)--(iii) are also equivalent to
\begin{itemize}
\item [(iv)] $C_{w,f}$ admits a dense absolutely mean irregular manifold.
\end{itemize}
\end{remark}

Let us now present some applications to weighted shifts and weighted translation operators. 

\begin{corollary}\label{MLY-UWBS}
For weighted shifts $B_w$ on $\ell^p(\N)$ and weighted translation operators $T_w$ on $L^p[1,\infty)$,
properties (i)--(iv) above are always equivalent to each other.
\end{corollary}

\begin{proof}
The conditions in Remark~\ref{R-DMLY-PB}(b) are satisfied.
\end{proof}

\begin{corollary}\label{MLY-BWBS}
A weighted shift $B_w$ on $\ell^p(\Z)$ with nonzero weights is mean Li-Yorke chaotic if and only if 
it is not absolutely Ces\`aro bounded and
\begin{equation}\label{MLY-eq4}
\liminf_{N \to \infty} \frac{1}{N} \sum_{n=1}^N |w_{-n} \cdots w_{-1}| = 0.
\end{equation}
\end{corollary}

\begin{proof}
($\Rightarrow$): Since $B_w$ is mean Li-Yorke chaotic, it is not absolutely Ces\`aro bounded.
Let $(B_i)_{i \in I}$ be the family given by Theorem~\ref{MLY-Nec}. Choose $i \in I$ and $j \in B_i$. 
By (\ref{MLY-eq1}),
\[
\liminf_{N \to \infty} \frac{1}{N} \sum_{n=1}^N |w_{-n+j} \cdots w_{-1+ j}| 
  = \liminf_{N \to \infty} \frac{1}{N} \sum_{n=1}^N \mu_n(f^{-n}(\{j\}))^\frac{1}{p} = 0,
\]
which gives (\ref{MLY-eq4}).

\smallskip\noindent
($\Leftarrow$): If
\[
\limsup_{N \to \infty} \frac{1}{N} \sum_{n=1}^N |w_{-n} \cdots w_{-1}| > 0,
\]
then $e_0$ is an absolutely mean irregular vector for $B_w$, and so $B_w$ is mean Li-Yorke chaotic. If
\[
\lim_{N \to \infty} \frac{1}{N} \sum_{n=1}^N |w_{-n} \cdots w_{-1}| = 0,
\]
then
\[
\lim_{N \to \infty} \frac{1}{N} \sum_{n=1}^N \mu_n(f^{-n}(B))^\frac{1}{p} = 0 \ \text{ for all } B \subset \Z \text{ finite}.
\]
Since $B_w$ is not absolutely Ces\`aro bounded, we can apply Theorem~\ref{DMLY} and conclude that $B_w$ is mean Li-Yorke chaotic.
\end{proof}

%%%%%%%%%%%%%%%%%%%%%%%%%%%%%%
%%%%%%%%%%%%%%%%%%%%%%%%%%%%%%

\subsection{The case of the space $C_0(\Omega)$}

\begin{theorem}\label{MLY-C0X}
Given a weighted composition operator $C_{w,f}$ on $C_0(\Omega)$, 
let $\fX$ be the closed subspace of $C_0(\Omega)$ generated by the functions $\varphi \in C_0(\Omega)$ 
whose support is contained in a relatively compact open set $B$ in $\Omega$ satisfying
\begin{equation}\label{MLY-C0X-1}
\liminf_{N \to \infty} \frac{1}{N} \sum_{n=1}^N \|w^{(n)}\|_{f^{-n}(B)} = 0.
\end{equation}
Then, the following assertions are equivalent:
\begin{itemize}
\item [\rm (i)] $C_{w,f}$ is mean Li-Yorke chaotic;
\item [\rm (ii)] There exists a sequence $(B_i)_{i \in \N}$ of relatively compact open sets in $\Omega$ such that
  \begin{equation}\label{MLY-C0X-2}
  \ov{B_i} \subset B_{i+1} \ \text{ for all } i \in \N,
  \end{equation}
  \begin{equation}\label{MLY-C0X-3}
  \liminf_{N \to \infty} \frac{1}{N} \sum_{n=1}^N \|w^{(n)}\|_{f^{-n}(B_i)} = 0 \ \text{ for all } i \in \N
  \end{equation}
  and
  \begin{equation}\label{MLY-C0X-4}
  \sup\Big\{\frac{1}{N} \sum_{n=1}^N \|w^{(n)}\|_{f^{-n}(B_i)} : i, N \in \N\Big\} = \infty.
  \end{equation}
\item [\rm (iii)] $C_{w,f}$ is not absolutely Ces\`aro bounded in $\fX$.
\end{itemize}
\end{theorem}

\begin{proof}
(i) $\Rightarrow$ (ii): Let $\varphi \in C_0(\Omega)$ be an absolutely mean irregular vector for $C_{w,f}$.
Then, there is an increasing sequence $(N_j)_{j \in \N}$ of positive integers such that
\[
\lim_{j \to \infty} \frac{1}{N_j} \sum_{n=1}^{N_j} \|(C_{w,f})^n(\varphi)\| = 0.
\]
For each $i \in \N$, consider the relatively compact open subset $B_i$ of $\Omega$ given by
\[
B_i = \{x \in \Omega : |\varphi(x)| > i^{-1}\}.
\]
Clearly, (\ref{MLY-C0X-2}) holds. Since, for each $i \in \N$,
\[
\frac{1}{N_j} \sum_{n=1}^{N_j} \|(C_{w,f})^n(\varphi)\|
  \geq \frac{1}{N_j} \sum_{n=1}^{N_j} \|(\varphi \circ f^n) \cdot w^{(n)}\|_{f^{-n}(B_i)}
  \geq \frac{1}{i} \frac{1}{N_j} \sum_{n=1}^{N_j} \|w^{(n)}\|_{f^{-n}(B_i)},
\]
we obtain (\ref{MLY-C0X-3}). Now, suppose that (\ref{MLY-C0X-4}) is false, that is,
\[
C = \sup\Big\{\frac{1}{N} \sum_{n=1}^N \|w^{(n)}\|_{f^{-n}(B_i)} : i, N \in \N\Big\} < \infty.
\]
For each $n \in \N$, take $x_n \in \Omega$ such that $\|(C_{w,f})^n(\varphi)\| = |\varphi(f^n(x_n)) w^{(n)}(x_n)|$.
Given $N \in \N$, there exists $i_N \in \N$ such that $x_n \in f^{-n}(B_{i_N})$ for all $n \in \{1,\ldots,N\}$. Hence,
\[
\frac{1}{N} \sum_{n=1}^N \|(C_{w,f})^n(\varphi)\| 
  = \frac{1}{N} \sum_{n=1}^N |\varphi(f^n(x_n)) w^{(n)}(x_n)| 
  \leq \frac{1}{N} \sum_{n=1}^N \|w^{(n)}\|_{f^{-n}(B_{i_N})} \|\varphi\|
  \leq C \|\varphi\|,
\]
which contradicts the fact that $\varphi$ is an absolutely mean irregular vector for $C_{w,f}$.

\smallskip\noindent
(ii) $\Rightarrow$ (iii): Without loss of generality, we may assume that $B_i \neq \emptyset$ for all $i \in \N$.
By (\ref{MLY-C0X-2}), for each $i \in \N$, there exists a continuous map $\phi_i : \Omega \to [0,1]$ such that 
$\supp \phi_i \subset B_{i+1}$ and $\phi_i = 1$ on $\overline{B_i}$.
By (\ref{MLY-C0X-3}), $\phi_i \in \fX$ for all $i \in \N$. Since $\|\phi_i\| = 1$ and 
\[
\frac{1}{N} \sum_{n=1}^N \|(C_{w,f})^n(\phi_i)\| 
  \geq \frac{1}{N} \sum_{n=1}^N \|(\phi_i \circ f^n) \cdot w^{(n)}\|_{f^{-n}(B_i)} 
  = \frac{1}{N} \sum_{n=1}^N \|w^{(n)}\|_{f^{-n}(B_i)},
\]
it follows from (\ref{MLY-C0X-4}) that $C_{w,f}$ is not absolutely Ces\`aro bounded in $\fX$.

\smallskip\noindent
(iii) $\Rightarrow$ (i): Suppose, by contradiction, that $C_{w,f}$ is not mean Li-Yorke chaotic. 
If $\phi \in C_0(\Omega)$ has support contained in a relatively compact open set $B$ in $\Omega$ satisfying (\ref{MLY-C0X-1}), then
\begin{align*}
\liminf_{N \to \infty} \frac{1}{N} \sum_{n=1}^N \|(C_{w,f})^n(\phi)\| 
  &= \liminf_{N \to \infty} \frac{1}{N} \sum_{n=1}^N \|(\phi \circ f^n) \cdot w^{(n)}\|_{f^{-n}(B)}\\
  &\leq \|\phi\| \liminf_{N \to \infty} \frac{1}{N} \sum_{n=1}^N \|w^{(n)}\|_{f^{-n}(B)} = 0.
\end{align*}
Since $C_{w,f}$ does not admit an absolutely mean semi-irregular vector, we conclude that
\[
\lim_{N \to \infty} \frac{1}{N} \sum_{n=1}^N \|(C_{w,f})^n(\phi)\| = 0.
\]
This implies that the set $\cR_1$ of all $\varphi \in \fX$ such that
\[
\liminf_{N \to \infty} \frac{1}{N} \sum_{n=1}^N \|(C_{w,f})^n(\varphi)\| = 0
\]
is residual in $\fX$. Now, by arguing as at the end of the proof of Theorem~\ref{MLY}, we get a contradiction.
\end{proof}

\begin{theorem}\label{DMLY-C0X}
Consider a weighted composition operator $C_{w,f}$ on $C_0(\Omega)$. If
\[
\liminf_{N \to \infty} \frac{1}{N} \sum_{n=1}^N \|w^{(n)}\|_{f^{-n}(B)} = 0
\]
for every relatively compact open set $B$ in $\Omega$, then the following assertions are equivalent:
\begin{itemize}
\item [\rm (i)] $C_{w,f}$ is mean Li-Yorke chaotic;
\item [\rm (ii)] $C_{w,f}$ has a residual set of absolutely mean irregular vectors;
\item [\rm (iv)] $C_{w,f}$ is not absolutely Ces\`aro bounded.
\end{itemize}
\end{theorem}

\begin {proof} 
Given $\varphi \in C_c(\Omega)$, take a relatively compact open set $B$ in $\Omega$ with $\supp \varphi \subset B$. 
By the hypothesis,
\[
\liminf_{N \to \infty} \frac{1}{N} \sum_{n=1}^N \|(C_{w,f})^{n}(\varphi)\| 
  \leq \liminf_{N \to \infty} \frac{1}{N} \sum_{n=1}^N \|w^{(n)}\|_{f^{-n}(B)} \|\varphi\| = 0.
\]
Since $C_c(\Omega)$ is dense in $C_0(\Omega)$, the equivalences between conditions (i), (ii) and (iii) follow from \cite[Theorem~22]{BBP}.
\end{proof}

\begin{remark}\label{R-DMLY-PB-C0X}
{\bf (a)} In the case $C_0(\Omega)$ is separable, \cite[Theorem~17]{BBP} says that (ii) is equivalent~to
\begin{itemize}
\item [(ii')] $C_{w,f}$ is densely mean Li-Yorke chaotic.
\end{itemize}

\noindent
{\bf (b)} If the space $C_0(\Omega)$ is separable and 
\[
\lim_{N \to \infty} \frac{1}{N} \sum_{n=1}^N \|w^{(n)}\|_{f^{-n}(B)} = 0
\]
for every relatively compact open set $B$ in $\Omega$, then \cite[Theorem~29]{BBP} implies that (i)--(iii) are also equivalent to
\begin{itemize}
\item [(iv)] $C_{w,f}$ admits a dense absolutely mean irregular manifold.
\end{itemize}
\end{remark}

Let us now present some applications to weighted shifts and weighted translation operators.

\begin{corollary}\label{MLY-UWBS-C0X}
For weighted shifts $B_w$ on $c_0(\N)$ and weighted translation operators $T_w$ on $C_0[1,\infty)$,
properties (i)--(iv) above are always equivalent to each other.
\end{corollary}

\begin{proof}
The conditions in Remark~\ref{R-DMLY-PB-C0X}(b) are satisfied.
\end{proof}

\begin{corollary}\label{MLY-BWBS-C0X}
A weighted shift $B_w$ on $c_0(\Z)$ with nonzero weights is mean Li-Yorke chaotic if and only if 
it is not absolutely Ces\`aro bounded and
\begin{equation}\label{MLY-C0X-5}
\liminf_{N \to \infty} \frac{1}{N} \sum_{n=1}^N |w_{-n} \cdots w_{-1}| = 0.
\end{equation}
\end{corollary}

\begin{proof}
The proof is analogous to that of Corollary~\ref{MLY-BWBS}, 
but we have to use Theorems~\ref{MLY-C0X} and~\ref{DMLY-C0X} instead of Theorems~\ref{MLY} and~\ref{DMLY}.
\end{proof}

\begin{corollary}
Given a weighted translation operator $T_w$ on $C_0(\R)$, let $\fX$ be the closed subspace of $C_0(\R)$ generated by the 
functions $\varphi \in C_0(\R)$ whose support is contained in a bounded open set $B$ in $\R$ satisfying
\[
\liminf_{N \to \infty} \frac{1}{N} \sum_{n=1}^N \Big(\sup_{x \in B} |w(x-n) \cdots w(x-1)|\Big) = 0.
\]
Then, the following assertions are equivalent:
\begin{itemize}
\item [\rm (i)] $T_w$ is mean Li-Yorke chaotic;
\item [\rm (ii)] There exists a sequence $(B_i)_{i \in \N}$ of bounded open sets in $\R$ such that
  \[
  \ov{B_i} \subset B_{i+1} \ \text{ for all } i \in \N,
  \]
  \[
  \liminf_{N \to \infty} \frac{1}{N} \sum_{n=1}^N \Big(\sup_{x \in B_i} |w(x-n) \cdots w(x-1)|\Big) = 0 \ \text{ for all } i \in \N
  \]
  and
  \[
  \sup\Big\{\frac{1}{N} \sum_{n=1}^N \Big(\sup_{x \in B_i} |w(x-n) \cdots w(x-1)|\Big) : i, N \in \N\Big\} = \infty.
  \]
\item [\rm (iii)] $T_w$ is not absolutely Ces\`aro bounded in $\fX$.
\end{itemize}
\end{corollary}

\begin{proof}
This is a particular case of Theorem~\ref{MLY-C0X}.
\end{proof}

%%%%%%%%%%%%%%%%%%%%%%%%%%%%%%%%%%%%%%%%%%%%%%%%%%%%%%%%%%%%%%%%%%%%%%
%%%%%%%%%%%%%%%%%%%%%%%%%%%%%%%%%%%%%%%%%%%%%%%%%%%%%%%%%%%%%%%%%%%%%%

\section*{Acknowledgements}

The first author was partially supported by CNPq -- Project {\#}308238/2021-4 and by CAPES -- Finance Code 001.
The second author was partially supported by Project PID2022-139449NB-I00, funded by MCIN/AEI/10.13039/501100011033/FEDER, UE.

%%%%%%%%%%%%%%%%%%%%%%%%%%%%%%%%%%%%%%%%%%%%%%%%%%%%%%%%%%%%%%%%%%%%%%
%%%%%%%%%%%%%%%%%%%%%%%%%%%%%%%%%%%%%%%%%%%%%%%%%%%%%%%%%%%%%%%%%%%%%%

\smallskip

{\footnotesize

\bigskip\noindent
{\sc Nilson C. Bernardes Jr.}

\smallskip\noindent
Departamento de Matem\'atica Aplicada, Instituto de Matem\'atica, Universidade Federal do Rio de Janeiro, 
Caixa Postal 68530, RJ 21941-909, Brazil.\\
\textit{ e-mail address}: ncbernardesjr@gmail.com

\bigskip\noindent {\sc Antonio Bonilla}

\smallskip\noindent
Departamento de An\'alisis Matem\'atico and Instituto de Matem\'aticas y Aplicaciones (IMAULL), 
Universidad de La Laguna, C/Astrof\'{\i}sico Francisco S\'anchez, s/n, 38721 La Laguna, Tenerife, Spain\\
\textit{ email address:} abonilla@ull.edu.es

\bigskip\noindent
{\sc Jo\~ao V. A. Pinto}

\smallskip\noindent
Departamento de Matem\'atica Aplicada, Instituto de Matem\'atica, Universidade Federal do Rio de Janeiro, 
Caixa Postal 68530, RJ 21941-909, Brazil.\\
\textit{ e-mail address}: joao.pinto@ufu.br

}

\end{document}